\documentclass[reqno,12pt,letterpaper]{amsart}

\usepackage{amsmath,amssymb,amsthm,graphicx,mathrsfs,url}
\usepackage{color}
\usepackage[colorlinks=true,linkcolor=red,citecolor=green]{hyperref}
\usepackage{amsxtra}
\usepackage{graphicx}
\usepackage{tikz}
\relax

\numberwithin{equation}{section}

\newcommand{\be}{\begin{equation}}
\newcommand{\ee}{\end{equation}}

\newcommand{\ov}{\overline}
\newcommand{\ud}{\underline}

\newcommand{\D}{{\mathcal D}}

\newcommand{\Div}{{\mathrm{div}}}

\newcommand{\op}{{\mathrm{Op}}}

\theoremstyle{plain}

\newtheorem{thm}{Theorem}
\newtheorem{prop}{Proposition}[section]

\newtheorem{lem}[prop]{Lemma}
\newtheorem{definition}[prop]{Definition}

\theoremstyle{definition}

\newtheorem{rem}[prop]{Remark}

\numberwithin{equation}{section}

\def\squarebox#1{\hbox to #1{\hfill\vbox to #1{\vfill}}}

\newcommand{\noi}{\noindent}

\usepackage{amsxtra}

\ifx\pdfoutput\undefined
\DeclareGraphicsExtensions{.pstex, .eps}
\else
\ifx\pdfoutput\relax
\DeclareGraphicsExtensions{.pstex, .eps}
\else
\ifnum\pdfoutput>0
\DeclareGraphicsExtensions{.pdf}
\else
\DeclareGraphicsExtensions{.pstex, .eps}
\fi
\fi
\fi

\title[Kelvin-Voigt damped wave equation]
{Decay for the Kelvin-Voigt damped wave equation: Piecewise smooth damping}
\author[N. Burq]{Nicolas Burq}
\address{Universit\'e  Paris-Saclay, Laboratoire de math\'ematiques d'Orsay, UMR 8628 du CNRS, B\^at. 307, 91405 Orsay Cedex, France and Institut Universitaire de France}
\email{nicolas.burq@universite-paris-saclay.fr}
\author[C-M. Sun]{Chenmin Sun}
\address{Universit\'e de Cergy-Pontoise, Laboratoire de Math\'ematiques AGM, UMR  8088 du CNRS, 2 av. Adolphe Chauvin
95302 Cergy-Pontoise Cedex, France }
\email{chenmin.sun@u-cergy.fr}

\usepackage{amssymb}
\usepackage{amsmath, amsthm, amsopn, amsfonts}

\def\11{{\rm 1~\hspace{-1.4ex}l} }
\def\R{\mathbb R}

\def\Z{\mathbb Z}
\def\N{\mathbb N}

\begin{document}

	\begin{abstract}
	We study the energy decay rate of the Kelvin-Voigt damped wave equation with piecewise smooth damping on the multi-dimensional domain. Under suitable geometric assumptions on the support of the damping, we obtain the optimal polynomial decay rate which turns out to be different from the one-dimensional case studied in \cite{LR05}. This optimal decay rate is saturated by  high energy quasi-modes localised on geometric optics rays  which hit the interface along non orthogonal neither tangential directions. The proof uses semi-classical analysis of boundary value problems.
	\end{abstract}   
	
	\maketitle 
	\setlength{\parskip}{0.3em}  
	
	\section{Introduction}   
	\label{in}
\subsection{Kelvin-Voigt damped wave equation}
In this article, we study the decay rate of the Kelvin-Voigt damped wave equation on the multi-dimensional bounded domain $\Omega\subset \R^d$, $d\geq 2$:
	\begin{equation}\label{KVDW-main}
\left\{
\begin{aligned}
	&\big(\partial_t^2-\Delta-\mathrm{div}a(x)\nabla\partial_t\big)u=0,\quad (t,x)\in\R_t\times\Omega,\\
	&u(t,\cdot)|_{\partial\Omega}=0,\\
	& (u,\partial_tu)|_{t=0}=(u_0,u_1)
\end{aligned}
\right.
\end{equation}
The damping $a(x)\geq 0$ is assumed to be piecewise smooth.
Denote by $\mathcal{H}^1=H_0^1\times L^2$. The solution of \eqref{KVDW-main} can be written as  
$$ U(t)=\binom{u(t)}{\partial_tu(t)}=e^{t\mathcal{A}}\binom{u_0}{u_1},
$$
where the generator
\begin{equation}\label{A}
\mathcal{A}=\left(\begin{matrix}
0   &1    \\
\Delta   &\mathrm{div}a\nabla
\end{matrix}\right)  
\end{equation}
with domain
$$ \mathcal{D}(\mathcal{A})=\{(u_0,u_1)\in H_0^1\times L^2: \Delta u_0+\mathrm{div} a\nabla u_1\in L^2, u_1\in H_0^1  \}.
$$
Note that the energy
$$ E[u](t)=\frac{1}{2}\|e^{t\mathcal{A}}(u_0,u_1)\|_{\mathcal{H}^1}^2 =\frac{1}{2}\int_{\Omega}\big(|\partial_tu|^2+|\nabla u|^2 \big)dx
$$
satisfies 
$$ E[u] (t) - E[u] (0) = - \int_{0}^t\int_{\Omega} a(x) |\nabla_x \partial_t u |^2 (s,x) ds$$

It was proved in \cite{BC15} and \cite{BS} (see also \cite{LR06},\cite{Te16} for related results) that if $a$ is smooth, vanishing nicely and the region $\{x\in\Omega: a(x)>0\}$ controls geometrically $\Omega$, then the rate of decay of the energy is exponential:
$$ E[u](t)\leq C\mathrm{e}^{-ct}E[u](0).
$$ 
In this article, we investigate the different case where the damping $a(x)$ is piecewise smooth and has a jump across some hypersurface $\Sigma\subset\Omega$. Unlike the smooth damping vanishing nicely, the problem with piecewise damping can be seen as an elliptic-hyperbolic transmission system on the two sides of the interface $\Sigma$ connected by some transmission condition. The interface becomes a wall to reduce the energy transmission from the hyperbolic region to the damped region. This phenomenon is known as {\em overdamping}. It turns out that this discontinuous Kelvin-Voigt damping $\nabla\cdot(a(x)\nabla\partial_t u)$ does not follow the principle that the ``geometric control condition" implies the exponential stabilization, which holds for the wave equation with localized viscous damping $a(x)\partial_tu$ (see~\cite{LR05, Zh18} for results using the multiplier methods)

\subsection{The main result}
To state our main result, we first make some geometric assumptions.
Let $\Omega\subset \R^d$ with $d\geq 2$. We consider the piecewise smooth damping $a\in C^{\infty}(\ov{\Omega}_1)$, $a|_{\Omega\setminus\Omega_1}=0$, such that there exists $\alpha_0>0$,
$$ \inf_{x\in\partial\Omega_1}a(x)\geq \alpha_0,
$$
 where $\Omega_1\subset\Omega.$ We assume that $\partial\Omega_1$ consists of $\partial\Omega$ and $\Sigma=\partial\Omega_1\setminus\partial\Omega$ where $\Sigma\subset \Omega$. Denote by $\Omega_2=\Omega\setminus (\Omega_1\cup \Sigma)$, then $\partial\Omega_2=\Sigma$ is the interface. We will fix this geometry in this article and assume that $\Omega_1, \Omega_2$ and $\Sigma$ are smooth ($C^\infty$, though this assumption could be relaxed to a finite number of derivatives).
\begin{center}
\begin{tikzpicture}
\draw[fill=blue!40] (0,0) ellipse[x radius=3.4cm, y radius =2.4 cm];
\draw[black] (0,1.8) node {$\Omega_1:a(x)\geq \alpha_0$};
\draw[fill=green!40] (-0.15,0) ellipse [x radius = 1.4cm, y radius =1.6cm];
\draw[black] (0,-0.4) node[above] {$\Omega_2:a(x)=0$};
\draw [->] (-0.15,-1.6) -- (-0.15,-0.8);
\draw[black] (0.2,-0.9) node{$\nu$};
\draw[black] (-0.6,-1.7) node{$\Sigma$};
\draw[black] (0,-3) node {
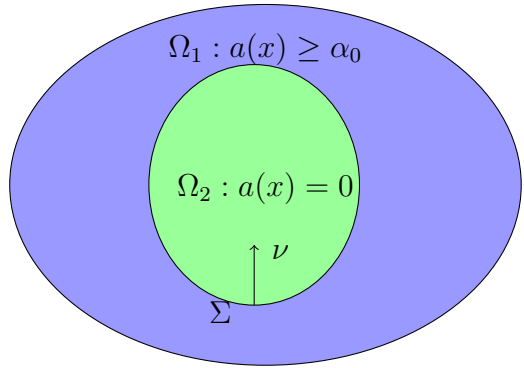
\text{Geometry of the damped region}};
\end{tikzpicture}
\end{center}

\begin{definition}[Geometric control condition]\label{GCC}
We say that $\Omega_1$ satisfies the geometric control condition, if all generalized rays (geometric optics reflecting on the boundary $\partial\Omega$ according to the laws of geometric optics) of $\Omega$ eventually reach the set $\Omega_1$ in finite time.	
\end{definition}

An alternative (equivalent in this context) property is the following
\begin{itemize}
	\item[(H)] All the bicharacteristics of $\Omega_2$ will reach a non-diffractive point (with respect to the domain $\Omega_2$) at the boundary $\Sigma$.
\end{itemize} 

\begin{thm}\label{thm:decay}
Assume that $\Omega,\Omega_1,\Omega_2$ and $a(x)$ satisfy the above geometric conditions. Then under the hypothesis (H),  there exists a uniform constant $C>0$, such that for every $(u_0,u_1)\in\mathcal{D}(\mathcal{A})$ and $t\geq 0$, 
\begin{equation}\label{eq:decayrate}
\|e^{t\mathcal{A}}(u_0,u_1)\|\leq \frac{C}{1+t}\|(u_0,u_1)\|_{\mathcal{D}(\mathcal{A})}.
\end{equation}
Moreover, the decay rate is optimal in the following sense: when $\Omega\subset \R^d$, $d\geq 2$ and  $\Omega_2=\mathbb{D}\subset \Omega$ is a unit ball, $\Omega_1=\Omega\setminus\Omega_2$,  the semi-group $e^{t\mathcal{A}}$ associated with the damping $a(x)=\mathbf{1}_{\ov{\Omega}_1(x)}$ satisfies
\begin{align}\label{lowerbound}
\sup_{0\neq (u_0,u_1)\in \mathcal{D}(\mathcal{A}) }\frac{ \|e^{t\mathcal{A}}(u_0,u_1)\|_{\mathcal{H}^1}
 } {\|(u_0,u_1)\|_{\mathcal{D}(\mathcal{A})} }\geq \frac{C'}{1+t},
\end{align} 
for all $t\geq 0$, where $C'>0$ is a uniform constant.

\end{thm}
\begin{rem} In~\cite{B19}, under the geometric control condition, a weaker decay rate, namely $\frac 1 {\sqrt{1+t}}$ was achieved with a simpler and very robust general proof requiring much less rigidity on the geometric setting. Notice also that in dimension $1$, a stronger decay rate, namely $\frac 1 {{(1+t)^2}}$ is known to hold~\cite[Section 3, Example 1]{LR05}. It is hence remarquable that in higher dimensions we can construct examples of geometries where the $\frac{ 1} {(1+ t)}$ decay rate is saturated. This phenomenon is linked to the fact that in higher dimensions there exists sequences of eigenfunctions of the Laplace operator in $\Omega _2$ with Dirichlet boundary condidtions (or at least high order quasi-modes), with mass concentrated along rays which do not encounter the boundary at normal incidence (a fact which is clearly false in dimension $1$, seeing that in this case the incidence is always normal).

\begin{center}
	\begin{tikzpicture}
	\draw[fill=blue!40] (0,0) ellipse[x radius=3.4cm, y radius =2.4 cm];
	\draw[black] (0,1.8) node {$\Omega_1$};
	\draw[fill=green!40] (-0.15,0) ellipse [x radius = 1.4cm, y radius =1.6cm];
	\draw[black] (0,-0.4) node[above] {$\Omega_2$};
	\draw[->,blue](-1.3,0.7) -- (-0.15,1.6);
	\draw[->,blue] (-0.15,1.6)--(1.1,0.6);
	\draw[->,blue] (1.1,0.6)--(-0.2,-1.58);
	\draw[->,blue] (-0.2,-1.58)--
	(-1.2,0);
	\draw[->,blue] (-1.1,0.7)--(0,1.57);
	\draw[->,blue] (0,1.57)--(1,0.8);
	\draw[->,blue] (1,0.8)--(-0.4,-1.56);
	\draw[->,blue]
	(-0.4,-1.56)--(-1.4,0);
	\draw[black] (-0.6,-1.7) node{$\Sigma$};
	\draw[black] (0,-3) node {
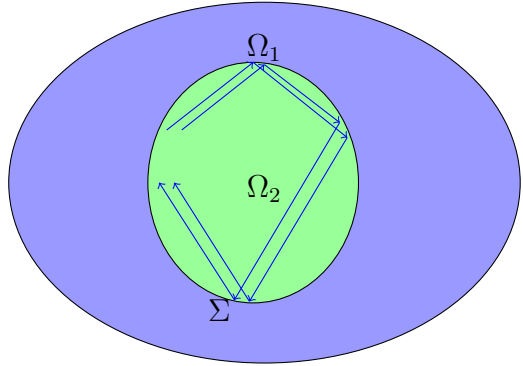
		\text{Angle of incidence is acute}};
	\end{tikzpicture}
\end{center}

\end{rem}

\begin{rem}
	Let us mention that the non-exponential stability for \eqref{KVDW-main} and a more general (theromo)viscoelastic system were studied in \cite{MRa}, where the authors obtained a rougher polynomial decay rate $O(t^{-\frac{1}{3}})$. Moreover, in our result, the damped region ($\Omega_1$) only needs to satisfy the geometric control condition, so the geometric configuration in Munoz Rivera-Racke is contained in our assumption.   
\end{rem}

\begin{rem} The choice of Dirichlet boundary conditions on $\partial \Omega$ plays no particular role, and we could have taken any type of boundary conditions for which the system is well posed and we have propagation of singularities (e.g. Neumann boundary conditions)
\end{rem}
\begin{rem}
The picture for Kelvin Voigt damping is now quite complete for smooth (essentially $C^2$) dampings \cite{BC15} and \cite{BS} (and also \cite{LR06},\cite{Te16}), or discontinuous dampings, see in dimension $1$ ~\cite[Section 3, Example 1]{LR05}, and the present paper. It would be interesting to understand the intermediate situation ($C^\alpha$, $\alpha \in (0, 2)$, dampings). We refer to~\cite{HZZ} for resuts in this direction in dimension $1$.
\end{rem}

\begin{rem}
In this article, we do not treat the case where $\Sigma\cap\partial\Omega\neq \emptyset$. In that case, $\partial\Omega_2$ can be only Lipchitz, and more technical treatments for the propagation of singularities are needed near the points $\Sigma\cap\partial\Omega$.
\end{rem}

\begin{thm}\label{thm:resolvent}
We have Spec$(\mathcal{A})\cap i\R=\emptyset$. Moreover, there exists $C>0$, such that for all $\lambda\in\R, |\lambda|\geq 1$,
	\begin{equation}\label{eq:resolvent}
	\big\|(i\lambda-\mathcal{A})^{-1}\big\|_{\mathcal{L}(\mathcal{H})}\leq C|\lambda|.
	\end{equation}
Moreover, when $\Omega\subset \R^d$, $d\geq 2$ and $\Omega_2=\mathbb{D}\subset \Omega$ is a unit ball, $\Omega_1=\Omega\setminus\Omega_2$, we actually have a lower bound:
$$\limsup_{\lambda\rightarrow + \infty} \lambda^{-1} \big\|(i\lambda-\mathcal{A})^{-1}\big\|_{\mathcal{L}(\mathcal{H})}=c >0.$$ In other words,  there exist sequences $(U_n)\subset  \mathcal{H}^1$ and $\lambda_n\rightarrow +\infty $ such that
\begin{equation}\label{lowerboundbis} \|U_n\|_{\mathcal{H}}=1,\; \|(i\lambda_n-\mathcal{A})U_n\|_{\mathcal{H}}=O(\lambda_n^{-1}).
\end{equation}
\end{thm}

\noi
$\bullet${ Theorem \ref{thm:decay} and  Theorem \ref{thm:resolvent}} are essentially equivalent. Indeed, the equivalence between the resolvent estimate \eqref{eq:resolvent} and the decay rate \eqref{eq:decayrate} is covered by Theorem 2.4 of \cite{BoTo}.  It is very likelly that~\eqref{lowerbound} and~\eqref{eq:resolvent} are also equivalent. However, we prove here only the fact that ~\eqref{lowerboundbis} implies \eqref{lowerbound}. We argue as follows: 
Let $U_n$ be a sequence of quasi-modes associated with $\lambda_n$ ($\lambda_n\rightarrow+\infty$) that saturates \eqref{eq:resolvent}. Denote by $F_n=(i\lambda_n-\mathcal{A})U_n$. We have
$$ \|U_n\|_{\mathcal{H}}=1,\;\|F_n\|_{\mathcal{H}}=O(\lambda_n^{-1}),\quad \|U_n\|_{\mathcal{D}(\mathcal{A})}\sim \lambda_n.
$$
Define $U_n(t)=\mathrm{e}^{t\mathcal{A}}U_n$ and we write
$$
 U_n(t)=\mathrm{e}^{i\lambda_n t}U_n+R_n(t),
$$ 
then
$$ (\partial_t-\mathcal{A})R_n=-(i\lambda_n-\mathcal{A})\mathrm{e}^{it\lambda_n}U_n=O_{\mathcal{H}}(\lambda_n^{-1}),\; R_n(0)=0.
$$
Since $$R_n(t)=-\int_0^t\mathrm{e}^{(t-s)\mathcal{A}}(i\lambda_n-\mathcal{A})\mathrm{e}^{is\lambda_n}U_nds,$$
we deduce that $\|R_n(t)\|_{\mathcal{H}}=O(\lambda_n^{-1}t)$ for $t>0$. Assume that $\kappa(t)$ is the optimal decay rate of the energy, then by $E[U_n(t)]^{\frac{1}{2}}=\|U_n(t)\|_{\mathcal{H}}\leq \kappa(t)^{\frac{1}{2}}\|U_n\|_{\mathcal{D}(\mathcal{A})}$ we have 
$$ C_1\kappa(t)^{\frac{1}{2}} \lambda_n\geq 1-\|R_n(t)\|_{\mathcal{H}}=1-C_2\lambda_n^{-1}t.
$$
For fixed $t>0$, we choose $n$ large enough such that $C_2\lambda_n^{-1}t=\frac{1}{2}$, thus we obtain that
$$ \kappa(t)^{\frac{1}{2}}\geq\frac{1}{2C_1\lambda_n}=\frac{1}{C_1C_2 t}.
$$
This proves \eqref{lowerbound}. As a consequence, we shall in the sequel reduce the analysis to the proof of Theorem~\ref{thm:resolvent}.

This article is organized as follows. We present the proof of \eqref{eq:resolvent} of Theorem \ref{thm:resolvent} in Section 2, Section 3 and Section 4. The proof follows from a contradiction argument which reduces the matter to study the associated high energy quasi-modes. In Section 2, we reduce the equation of quasi-modes to a transmission problem, consisting of an elliptic system in $\Omega_1$ and a hyperbolic system in $\Omega_2$, coupled at the interface $\Sigma$. Both systems are semi-classical but with different scales $h, \hbar= h^{1/2}$. Next in Section 3, we study the elliptic system and obtain the information of the quasi-modes restricted to the interface by transmission conditions. Then in Section 4, we prove the propagation theorem for the hyperbolic problem in $\Omega_2$ which will lead to a contradiction. We need to analyze two semi-classical scales corresponding to the elliptic and hyperbolic region, connected by the transmission condition on the interface.
 Finally in Section 5, we construct a sequence of quasi-modes saturating the inequality \eqref{eq:resolvent} in a simple geometry. In particular this proves the optimality of the resolvent estimate. We collect various toolboxes in the final section of the appendix. 
 
Throughout this article, we adopt the standard notations in semi-classical analysis (see for example \cite{EZB}). We will use the standard quantization for classical and semi-classical pseudo-differential operators $\mathrm{Op},\mathrm{Op}_h,\mathrm{Op}_{\hbar}$. We will also adopt the usual asymptotic notations, such as $O(h^{\alpha}),O(\hbar^{\alpha})$ and $o(h^{\alpha}),o(\hbar^{\alpha})$, as $h\rightarrow 0$. Moreover, for a Banach space $X$ and $h$-dependent families of functions $f_h,g_h$, we mean $f_h=O_X(h^{\alpha})$, $g_h=o_X(h^{\alpha})$, if
 $$ \|f_h\|_X=O(h^{\alpha}),\quad \|g_h\|_X=o(h^{\alpha}),
 $$
 as $h\rightarrow 0$.

\subsection*{Acknowledgment}
The first author is supported by Institut Universitaire de France and
		ANR grant ISDEEC, ANR-16-CE40-0013. The second author is supported by the postdoc programe: ``Initiative d'Excellence Paris Seine" of CY Cergy-Paris Universit\'e and ANR grant
ODA (ANR-18-CE40- 0020-01).

\section{Reduction to a transmission problem }

It was proved by the first author in \cite{B19} that 
$$ \|(\mathcal{A}-i\lambda)^{-1}\|_{\mathcal{L}(\mathcal{H})}\leq Ce^{c|\lambda|}
$$ 
under more general conditions for the damping. Therefore, the proof of the first part of Theorem \ref{thm:resolvent} (i.e. \eqref{eq:resolvent}) is reduced to the high energy regime $|\lambda|\rightarrow+\infty$.
  For this, we argue by contradiction. Assume that \eqref{eq:resolvent} is not true, then there exist $h$-dependent functions $U=\binom{u}{v}, F=\binom{f}{g}$, such that 
\begin{align}\label{seq:contradiction} 
&\|U_j\|_{H^1\times L^2}=O(1),\;\|F_j\|_{H^1\times L^2}=o(h)\\
& (\mathcal{A}-ih^{-1})U=F.
\end{align}

 Let $\nu$ be the unit normal vector pointing to the undamped region $\Omega$. Denote by $a_1(x)=a(x)\mathbf{1}_{\ov{\Omega}_1}$
Let $U=\binom{u}{v}$ and $F=\binom{f}{g}$. Then for $U\in D(\mathcal{A})$ and $F\in\mathcal{H}$, the equation
$$ (\mathcal{A}-i\lambda)U=F
$$
is equivalent to ($h=\lambda^{-1}$) the following system for $u_j=u\mathbf{1}_{\Omega_j}, f_j=f\mathbf{1}_{\Omega_j},$ and $g_j=g\mathbf{1}_{\Omega_j}$, $j=1,2$:
\begin{align}\label{system1}
\begin{cases} 
& u_1=ih(f_1-v_1), \text{ in }\Omega_1\\
& h\Delta u_1+h\nabla_x\cdot(a_1(x)\nabla v_1)-i v_1=hg_1,\text{ in } \Omega_1 \\
& u_2=ih(f_2-v_2), \text{ in } \Omega_2\\
& h\Delta u_2-iv_2=hg_2,\text{ in }\Omega_2
\end{cases}
\end{align}
with boundary condition on the interface
\begin{align}\label{BC:interface}
u_1|_{\Sigma}=u_2|_{\Sigma},\quad \partial_{\nu}u_2|_{\Sigma}=(\partial_{\nu}u_1+a_1\partial_{\nu}v_1 )|_{\Sigma},
\end{align}
Indeed, the equations inside $\Omega_1,\Omega_2$ can be verified directly.  The first boundary condition is just the fact that the function $u$ equal to $u_j$ in $\Omega_j$ must have no jump at the interface to enssure taht ity belongs to $H^1( \Omega)$. To check the second boundary condition, we take an arbitrary test function $\varphi\in C_c^{\infty}(\Omega)$ and multiply the equation $h\Delta u-iv+h\Div a\nabla v=0$ by $\varphi$. We obtain that
\begin{align*}
0=&-h\int_{\Omega}\nabla u\cdot\nabla\ov{\varphi}-h\int_{\Omega}a\nabla v\cdot\nabla\ov{\varphi}-i\int_{\Omega}v\ov{\varphi}\\
=&-\sum_{j=1}^2\int_{\Omega_j}\big(h\nabla u_j\cdot\nabla\ov{\varphi}-iv_j\ov{\varphi}\big)-h\int_{\Omega_1}a_1(x)\nabla v_1\cdot\nabla\ov{\varphi}\\
=&\sum_{j=1}^2\int_{\Omega_j}\big(h\Delta u_j\ov{\varphi}-iv_j\ov{\varphi}\big)+\int_{\Omega_1}h\nabla_x\cdot(a_1(x)\nabla v_1)\cdot \ov{\varphi}+h\int_{\Sigma}(\partial_{\nu}u_2-\partial_{\nu}u_1-a_1\partial_{\nu}v_1)|_{\Sigma}\cdot\ov{\varphi}.
\end{align*}
Using the differential equations in $\Omega_1,\Omega_2$, the last term on the right side is equal to $$h\int_{\Sigma}(\partial_{\nu}u_2-\partial_{\nu}u_1-a_1\partial_{\nu}v_1)|_{\Sigma}\cdot\ov{\varphi}|_{\Sigma},$$
hence it must vanish for all $\varphi$. This verifies \eqref{BC:interface}.

First we prove an a priori estimate for these functions:
\begin{lem}[A priori estimate]\label{apriori}
Denote by $U_j=\binom{u_j}{v_j}, F_j=\binom{f_j}{g_j}$, for $j=1,2$. 
Assume that $\|U_j\|_{H^1\times L^2}=O(1)$ and $\|F_j\|_{H^1\times L^2}=o(h)$, then we have
$$ \|\nabla v_1\|_{L^2}=o(h^{\frac{1}{2}}), \quad \|v_1\|_{L^2}=o(h)
$$
and
$$ \|\nabla u_1\|_{L^2}=o(h^{\frac{3}{2}}),\quad \|u_1\|_{L^2}=o(h^2).
$$
Consequently, by the trace theorem, we have
$$ \|u_1\|_{H^{\frac{1}{2}}(\Sigma)}=o(h^{\frac{3}{2}}),\quad \|v\|_{H^{\frac{1}{2}}(\Sigma)}=o(h^{\frac{1}{2}}).
$$
\end{lem}
\begin{proof}

First we observe that, from the relation between $u$ and $v$, we deduce that $\nabla v\in L^2(\Omega)$ and
\begin{equation}\label{eq1}
\|\nabla v_j\|_{L^2(\Omega_j)}=O(h^{-1}),\quad j=1,2.
\end{equation}
Moreover, by the trace theorem, $v_1|_{\Sigma}=v_2|_{\Sigma}$ as functions in $H^{\frac{1}{2}}(\Sigma)$.
From the system \eqref{system1}, we have
\begin{align}
& (\nabla u_1, \nabla v_1)_{L^2(\Omega_1)}\notag \\
&=ih(\nabla f_1,\nabla v_1)_{L^2(\Omega_1)}-ih\|\nabla v_1\|_{L^2(\Omega_1)}^2-(\nabla u_1,\nabla v_1)_{L^2(\Omega_1)}\notag \\
& \qquad -\|a_1^{1/2}\nabla v_1\|_{L^2(\Omega_1)}^2+(\partial_{\nu}u_1+a_1\partial_{\nu}v_1,v_1)_{L^2(\Sigma)} \label{eq2}\\
&=ih^{-1}\|v_1\|_{L^2(\Omega_1)}^2+(g_1,v_1)_{L^2(\Omega_1)}-(\nabla u_2, \nabla v_2)_{L^2(\Omega_2)}\label{eq3}\\
&=ih(\nabla f_2,\nabla v_2)_{L^2(\Omega_2)}-ih\|\nabla v_2\|_{L^2(\Omega_2)}^2-(\nabla u_2,\nabla v_2)_{L^2(\Omega_2)}-(\partial_{\nu}u_2,v_2)_{L^2(\Sigma)}\label{eq4}\\
&=ih^{-1}\|v_2\|_{L^2(\Omega_2)}^2+(g_2,v_2)_{L^2(\Omega_2)} \label{eq5}.
\end{align}
Taking the real part of \eqref{eq2}+\eqref{eq3}-\eqref{eq4}+\eqref{eq5}, we deduce that
$ \|\nabla v_1\|_{L^2(\Omega_1)}^2=o(h)
$, thanks to the boundary condition \eqref{BC:interface} and $v_1|_{\Sigma}=v_2|_{\Sigma}$. Therefore, from the first equation of \eqref{system1}, we have $\|\nabla u_1\|_{L^2(\Omega_1)}^2=o(h^3)$. Then, using this fact and the second equation of \eqref{system1}, we deduce that
$$ iv_1=h\Delta u_1+h\nabla\cdot(a_1\nabla v_1)-hg_1=O_{H^{-1}(\Omega_1)}(h^{\frac{3}{2}}).
$$
By interpolation, we have $v_1=o_{L^2(\Omega_1)}(h)$, and from $u_1=ih(f_1-v_1)$, $u_1=o_{L^2(\Omega_1)}(h^2)$. This completes the proof of Lemma \ref{apriori}.
\end{proof}

\section{Estimates of the elliptic system}\label{sectionelliptic}

\subsection{Standard theory}
We briefly recall the semiclassical elliptic boundary value problem near the interface $\Sigma$.  
 In what follows, we will sketch the parametrix construction for \eqref{elliptic1}, following \cite{BL03}. Near a point $x_0\in\Sigma$, we use the coordinate system $(y,x')$ where $\Omega_1=\{(y,x'):y>0\}$ near $x_0$. 
 \begin{equation}\label{elliptic1} L_{\hbar}w=\kappa=o_{L^2}(\hbar^2),\quad w|_{\Omega_1}=o_{H^1}(\hbar),\; w|_{\Sigma}=o_{H^{\frac{1}{2}}}(\hbar)
 \end{equation}
 where in the local coordinate chart, 
$$ L_{\hbar}:=\hbar^2D_y^2-R(y,x',\hbar D_{x'})+\sum_{j=1}^{d-1}\hbar M_j(y,x')\hbar\partial_{x_j'}+\hbar H(y,x')\hbar\partial_y.
$$
Here $R(y,x',\hbar D_{x'})$ is a second order semiclassical differential operator in $x'$ with the principal symbol $r(y,x',\xi')$. The principal symbol of
$L_{\hbar}$ is
$$ l(y,x',\eta,\xi')=\eta^2-r(y,x',\xi'),
$$
and we denote by
$$ m(y,x',\eta,\xi')=\sum_{j=1}^{d-1}M_j(y,x')\xi_j'+H(y,x')\eta. 
$$
The set of elliptic points in $T^*\partial\Omega$ is given by
$$ \mathcal{E}:=\{(y=0,x',\xi'): r(0,x',\xi')<0 \}
$$
By homogeneity, near a point $\rho_0\in\mathcal{E}$
\begin{equation}\label{localization-elliptic}
 -r(y,x',\xi')\geq c(\rho_0)|\xi'|^2.
\end{equation}
 Denote by $\ud{w}:=w\mathbf{1}_{y\geq 0}$ the extension by zero of $w$, and the same for $\ud{\kappa},$ etc. Then $\ud{w}$ satisfies the equation
\begin{equation}\label{elliptic2}
L_{\hbar}\ud{w}=-\hbar(\hbar\partial_yw)|_{y=0}\otimes\delta_{y=0}+\hbar^2w|_{y=0}\otimes\delta'_{y=0}+\hbar^2H(0,x')w|_{y=0}\otimes\delta_{y=0}+\ud{\kappa}.
\end{equation} 
Let $\varphi(y,x')$ be a cut-off to the local chart. Let $\psi\in C^{\infty}(\R^{d-1})$, be a Fourier multiplier in $S_{\xi'}^0$ such that on the support of $\varphi(y,x')\psi(\xi')$, \eqref{localization-elliptic} holds and $\varphi(y,x')\psi(\xi')=1$ near $\rho_0$. We define
\begin{equation}\label{symbol:elliptic}
 e^0(y,x',\eta,\xi'):=\frac{\varphi(y,x')\psi(\xi')}{l(y,x',\eta,\xi')}
\end{equation}
and $e^j,j\geq 1$ inductively by
\begin{equation*}
\begin{split}
e^1\cdot l=&-\sum_{|\alpha|=1}\frac{1}{i}\partial_{\xi',\eta}^{\alpha}e^0\cdot\partial_{x',y}^{\alpha}l-e^0\cdot m,\\
e^j\cdot l=-&\sum_{|\alpha|+k=n,k\neq n}\frac{1}{i^{|\alpha|}}\partial_{\xi',\eta}^{\alpha}e^k\cdot\partial_{x',y}^{\alpha}l-\sum_{|\alpha|+k=n-1}\frac{1}{i^{|\alpha|}}\partial_{\xi',\eta}^{\alpha}e^k\cdot\partial_{x',y}^{\alpha}m.
\end{split}
\end{equation*}
For any $N\in\N$, we define
$$ e_N=\sum_{j=0}^N\hbar^je^j,\quad E_N=\mathrm{Op}_{\hbar}(e_N),
$$
and then
$$  E_NL_{\hbar}=\varphi(y,x')\psi(\xi')\mathrm{Id}+R_N,
$$
where
$$ R_N=\mathcal{O}(\hbar^{N+1}): L_{x',y}^2\rightarrow L_{x',y}^2,\quad R_N=\mathcal{O}(\hbar^{N+1-2M}): H_{x',y}^s\rightarrow H_{x',y}^{s+2M},
$$
and
$$ E_N=\mathcal{O}(1):L_{x',y}^2\rightarrow L_{x',y}^2,\quad E_N=\mathcal{O}(\hbar^{-2}): H_{x',y}^s\rightarrow H_{x',y}^{s+2},
$$
thanks to Lemma \ref{technical1}.
Applying $E_N$ to the equation \eqref{elliptic2}, we obtain that
\begin{equation*}
\begin{split}
\varphi(y,x')\psi(\hbar D_{x'})\ud{w}=&-\hbar^2E_N((\partial_{y}w)|_{y=0}\otimes\delta_{y=0})+\hbar^2E_N(w|_{y=0}\otimes\delta'_{y=0} )+\hbar^2E_N(Hw|_{y=0}\otimes\delta_{y=0} )\\+&E_N\ud{\kappa}-R_N\ud{w}.
\end{split}
\end{equation*}
Note that $e_N(y,x',\eta,\xi')$ is meromorphic in $\eta$ with poles $\eta_{\pm}=\pm i\sqrt{-r(y,x',\xi')}$.
Denote by $G(x')=\partial_yw(0,x')+H(0,x')w(0,x')$, we calculate for $y>0, x'\in\R^{d-1}$ that
\begin{equation*}
\begin{split}
&\hbar^2E_N((\partial_yw+Hw)|_{y=0}\otimes\delta_{y=0})(y,x')\\=&\frac{\hbar^2}{(2\pi\hbar)^d}\int G(\widetilde{x}')\mathrm{e}^{\frac{i(x'-\widetilde{x}')\cdot\xi'}{\hbar}}d\widetilde{x}'d\xi'\int e_N(y,x',\eta,\xi')\mathrm{e}^{\frac{iy\eta}{\hbar}}d\eta\\
=&\frac{i\hbar}{(2\pi\hbar)^{d-1}}\int \mathrm{e}^{\frac{iy\eta_+}{\hbar}}n_N(y,x',\xi')\mathrm{e}^{\frac{i(x'-\widetilde{x'})\cdot\xi'}{\hbar}}G(\widetilde{x}')d\widetilde{x}'d\xi',
\end{split}
\end{equation*}
where $n_N(y,x',\xi')=\mathrm{Res}(e_N(y,x',\eta,\xi');\eta=\eta_+)$. Similarly, for $y>0, x'\in\R^{d-1}$,
\begin{equation*}
\begin{split}
\hbar^2E_N(w|_{y=0}\otimes\delta'_{y=0})(y,x')=&\frac{i\hbar}{(2\pi\hbar)^d}\int w(0,\widetilde{x}')\mathrm{e}^{\frac{i(x'-\widetilde{x}')\cdot\xi'}{\hbar}}d\widetilde{x}'d\xi'\int \eta\mathrm{e}^{\frac{iy\eta}{\hbar}}
e_N(y,x',\eta,\xi')d\eta\\
=&-\frac{1}{(2\pi\hbar)^{d-1}}\int \mathrm{e}^{\frac{iy\eta_+}{\hbar}}d_N(y,x',\xi')\mathrm{e}^{\frac{i(x'-\widetilde{x}')\cdot\xi'}{\hbar}}w(0,\widetilde{x}')d\widetilde{x}'d\xi',
\end{split}
\end{equation*}
where $d_N=\eta_+n_N$. Therefore,
\begin{multline}\label{pseudo-relation}
\varphi(y,x')\psi(\hbar D_{x'})\ud{w}\\
=i\mathrm{Op}_{\hbar}(\mathrm{e}^{iy\eta_+/\hbar}n_N(y,\cdot))\big(-(\hbar\partial_yw)|_{y=0}+\hbar (Hw)|_{y=0}\big)-\mathrm{Op}_{\hbar}(\mathrm{e}^{iy\eta_+/\hbar}d_N(y,\cdot))(w|_{y=0})\\+E_N\ud{\kappa}-R_N\ud{w},
\end{multline}
where the two operators in the expression above are tangential. Note that by Lemma \ref{technical1}
$$ R_N\ud{w},\quad E_N\ud{\kappa}=o_{L_{x',y}^2}(\hbar^2)=o_{H_{x',y}^2}(1),
$$
hence from the interpolation and the trace theorem, we have 
$$ (R_N\ud{w})|_{y=0}=o_{H_{x'}^{1/2}}(\hbar),\quad (E_N\ud{\kappa})|_{y=0}=o_{H_{x'}^{1/2}}(\hbar).
$$
Taking the trace $y=0$ for \eqref{pseudo-relation}, we obtain that
\begin{equation}\label{trace-relation} \mathrm{Op}_{\hbar}(\varphi(0,x')\psi(\xi')+d_N(0))(w|_{y=0})=-\mathrm{Op}_{\hbar}(in_N(0))((\hbar\partial_yw)|_{y=0}+\hbar(Hw)|_{y=0} )+o_{H_{x'}^{1/2}}(\hbar).
\end{equation}
Note that the principal symbols of $n_N(0), d_N(0)$ are
$$ \sigma(in_N(0))=\frac{\varphi(0,x')\psi(\xi')}{2\sqrt{-r(0,x',\xi')}},\quad \sigma(d_N(0))=\frac{\varphi(0,x')\psi(\xi')}{2}.
$$
In summary, there exists (near $\rho_0$)  a $\hbar$-P.d.O $\mathcal{N}_{\hbar}$, elliptic and of order 1 classic and of order 0 semi-classic, in the sense that
$$ \mathcal{N}_{\hbar}=\mathcal{O}(\hbar): H_{x'}^s\rightarrow H_{x'}^{s-1},
$$
such that 
$$ (\hbar\partial_yw)|_{y=0}=\mathcal{N}_{\hbar}(w|_{y=0}+ O_{H^{1/2}}(\hbar)).
$$
\subsection{Control of the semi-classical wave-front set of the trace}
For the further need, we should also control the wave front set of the precise elliptic equation (with $\hbar=h^{\frac{1}{2}}$) $$\hbar^2\Delta w-\frac{i}{a_1}w+\hbar\frac{\nabla a_1}{a_1}\cdot\hbar\nabla w=\kappa,
$$
where the $h$-semiclassical wave front set of the Neumann data
$ \mathrm{WF}_h(\partial_{\nu}w|_{\Sigma}).
$
 Here we need to pay attention to two different semi-classical scales.

\begin{prop}\label{control:frontdonde}
 Assume that $w$ satisfies the $\hbar$-semiclassical elliptic equation (with $\hbar=h^{\frac{1}{2}}$)
$$ \hbar^2\Delta w-\frac{i}{a_1}w+\hbar\frac{\nabla a_1}{a_1}\cdot\hbar\nabla w=\kappa
$$
with Neumann trace $\partial_{\nu}w|_{\Sigma}$ and $ \mathrm{WF}_h(\partial_{\nu}w|_{\Sigma})
$ is contained in a compact subset of $T^*\Sigma\setminus\{0\}$. Assume moreover that $w=O_{H^1}(h^{\frac{1}{2}})$ and $\kappa=O_{L^2}(h)$, then we have
$$ \mathrm{WF}_h(w|_{\Sigma})\subset \mathrm{WF}_h(\partial_{\nu}w|_{\Sigma})\cup \pi\big(\mathrm{WF}_h(\kappa)\big),
$$
where $\pi: T^*\Omega_1\rightarrow T^*\Sigma$ is the projection defined for points near $T^*\Sigma$, and 
$$ \pi\big(\mathrm{WF}_h(\kappa)\big)=\big\{\rho_0\in T^*\Sigma: \exists \rho\in T^*\Omega_1, \text{ near } T^*\Sigma, \text{ such that } \rho\in \mathrm{WF}_h(\kappa) \text{ and } \pi(\rho)=\rho_0 \big\}.
$$
\end{prop}
\begin{proof}
Let $(x_0,\xi_0)\notin \mathrm{WF}_h(\partial_{\nu}w|_{\Sigma})\cup\pi\big(\mathrm{WF}_h(\kappa)\big)$. Locally near $x_0\in\Sigma$, we can choose local coordinate system as in the previous subsection. Here the cutoff $\psi(\xi')$ can be chosen as $1$, since the operator $\hbar^2\Delta-i$ is always elliptic. Consider the tangential $h$-P.d.O  $A_h$ which is elliptic near $(x_0,\xi_0)$ and its principal symbol is supported away from $\mathrm{WF}_h(\partial_{\nu}w|_{\Sigma})\cup \pi\big(\mathrm{WF}_h(\kappa)\big)$. We need to show that $(A_hw)|_{y=0}=O_{L^2(\Sigma)}(h^{\infty})$. 

From \eqref{pseudo-relation}  we have
\begin{align*}
 \varphi(y,x')\ud{w}=&i\mathrm{Op}_{\hbar}\big(\mathrm{e}^{\frac{iy\eta_+}{\hbar}}n_N(y) \big)\big(-(\hbar\partial_yw)|_{y=0}+\hbar(Hw)|_{y=0} \big)-\mathrm{Op}_{\hbar}\big(\mathrm{e}^{\frac{iy\eta_+}{\hbar}}d_N(y) \big)(w|_{y=0})\\+&E_N\ud{\kappa}+O_{H^{1}}(h^{\frac{N}{2}}),
\end{align*}
where we gain $\hbar^N$ for $R_N\ud{w}$. By taking the trace $y=0$ and using the fact that $d_N(0)=\frac{1}{2}\varphi(0,x'),$ we obtain that
\begin{align*}
&(A_h\varphi(y,x')\ud{w})|_{y=0}+\big(A_h\mathrm{Op}_h(d_N(0))w\big)|_{y=0}-i\hbar (A_h\mathrm{Op}_{\hbar}(n_N)(Hw )|_{y=0}\\=&-iA_h\mathrm{Op}_{\hbar}(n_N(0))(\hbar\partial_{y}w)|_{y=0}+ (A_hE_N\ud{\kappa})|_{y=0}
+O_{H_{y=0}^{\frac{1}{2}}}(h^{N/2}).
\end{align*}
We claim that it suffices to show that
\begin{align}\label{WFh1} A_h\mathrm{Op}_{\hbar}(n_N(0))(\hbar\partial_{y}w)|_{y=0}=O_{L_{y=0}^2}(h^{\infty}) \text{ and } (A_hE_N\ud{\kappa})|_{y=0}=O_{L_{y=0}^2}(h^{\infty}).
\end{align}
Indeed, once this is done, we obtain that, at least $(A_h\ud{w})|_{y=0}=O_{L^2}(\hbar)$\footnote{The operator $A_h$ is of the form $\chi A_h\chi$ for some cutoff $\chi$ such that $\chi\equiv 1$ on supp$(\varphi)$. }.
Now we can replace $A_h$ by another tangential operator $\widetilde{A}_h$ with principal symbol $\widetilde{a}$ such that $\widetilde{a}$ is supported in a slightly larger region containing supp$(a)$ and $\widetilde{a}=1$ on supp$(a)$. We still have $(\widetilde{A}_h\ud{w})|_{y=0}=O_{L^2}(\hbar)$. Now we write
$$ \hbar A_h\mathrm{Op}_{\hbar}(n_N)Hw=\hbar A_h\mathrm{Op}_{\hbar}(n_N)H\widetilde{A}_hw+
\hbar A_h\mathrm{Op}_{\hbar}(n_N)H(1-\widetilde{A}_h)w.
$$
From Lemma \ref{commutator:2echelles}, the trace of the second term on the right side is $\mathcal{O}_{L^2}(h^{\infty})$. Therefore, the trace of the first term on the right side is $O_{L^2}(\hbar^2)$, hence $(A_h\ud{w})|_{y=0}=O_{L^2}(\hbar^2)$. Then we can continuously apply this argument to conclude.  

It remains to prove \eqref{WFh1}. For this, we just need to interchange the operator $A_h$ with $E_N$ and $\mathrm{Op}_{\hbar}(n_N(0))$. 
Here additional attentions are needed, since $\mathrm{Op}_{\hbar}(n_N(0)), E_N$ are $\hbar$-P.d.O. This can be verified from the following lemma:
\begin{lem}\label{commutator:2echelles}
 Assume that $a,b,q\in S^0(\R_x^n\times\R_{\xi}^n)$, compactly supported in the $x$ variable such that
$$ \mathrm{dist}\big(\mathrm{supp}(a),\mathrm{supp}(b) \big)\geq c_0>0.
$$
 Then for any $s\in\R, N\in\N, N\geq 2n$, we have
$$ a(x,hD_x)q(x,h^{\frac{1}{2}}D_x) b(x,hD_x)=\mathcal{O}_{L^2\rightarrow L^2}(h^{N}).
$$
\end{lem}
\begin{proof}
Denote by
$$ A(x,y,\xi,\eta)=a(x,h(\xi+\eta))q(x+y,h^{\frac{1}{2}}\xi).
$$
Then from Lemma \ref{symbolic}, 
\begin{align*}
a(x,hD_x)q(x,h^{\frac{1}{2}}D_x)=&\sum_{|\beta|\leq N}\mathrm{Op}\big(\frac{h^{|\beta|}}{i^{|\beta|}\beta!}(\partial_{\xi}^{\beta}a)(x,h\xi)(\partial_x^{\beta}q)(x,h^{\frac{1}{2}}\xi) \big)\\
+&O_{\mathcal{L}(L^2)}(h^{N+1-n}),
\end{align*}
since for any $\beta\in \N^{2n}$,
$$ \sup_{|\alpha|=N+1}\sup_{(x,\xi)}\iint_{\R^{2n}}|\partial_{x,\xi}^{\beta}\partial_z^{\alpha}\partial_{\zeta}^{\alpha}(a(x,h(\xi+\zeta)) q(x+z,h^{\frac{1}{2}}\xi) ) |dz d\zeta =O(h^{N+1-n}).
$$
 Using the fact that $\frac{h^{|\beta|}}{i^{|\beta}|\beta!}\partial_{\xi}^{\beta}a\cdot \partial_x^{\beta}q\cdot b=0$, thanks to the support property, we have, using again Lemma \ref{symbolic}, 
$$ a(x,hD_x)q(x,h^{\frac{1}{2}}D_x)b(x,hD_x)=O_{\mathcal{L}(L^2)}(h^N)
$$
for any $N$ large enough. This completes the proof.
\end{proof}
Therefore the proof of Proposition \ref{control:frontdonde} is now complete.
\end{proof}

\subsection{Estimate of the traces}
Let $u_1,v_1$ be solutions of the first two equations of \eqref{system1}. Consider $w=u_1+a_1v_1$, then under the assumption of Lemma \ref{apriori},
$$ w=o_{H^1}(h^{\frac{1}{2}}),\quad w=o_{L^2}(h),\quad w|_{\Sigma}=o_{H^{\frac{1}{2}}}(h^{\frac{1}{2}}).
$$
Note that $w$ satisfies the elliptic equation (with $\hbar=h^{\frac{1}{2}}$)
\begin{equation}\label{w:elliptic}
\begin{split}
\hbar^2\Delta w+\hbar\frac{\nabla a_1}{a_1}\cdot \hbar\nabla w-\frac{i}{a_1}w=\hbar^2g_1-\hbar^2\Delta a_1\cdot v_1+\hbar^2\frac{|\nabla a_1|^2}{a_1}v_1-\frac{\hbar \nabla a_1}{a_1}\cdot\hbar\nabla u_1+\frac{i}{a_1}u_1
\end{split}
\end{equation}

In particular,
$$ \hbar^2\Delta w-\frac{i}{a_1}w+\hbar\frac{\nabla a_1}{a_1}\cdot \hbar\nabla w=o_{L^2}(\hbar^2).
$$
In this case, $\mathcal{N}_{\hbar}$ defined in the last subsection is the usual $\hbar$-semiclassical Dirichlet-Neumann operator:
$$ \mathcal{N}_{\hbar}(w|_{\Sigma}+ o_{H^{1/2}} ( \hbar^2)):=(\hbar\partial_{\nu}w)|_{\Sigma}.
$$ We can apply the standard theory (to $h^{-1} w$)with the particular choice $\psi(\xi')\equiv 1$ in \eqref{symbol:elliptic} and obtain the following:
\begin{prop}\label{estimatetrace1}
Let $\chi\in C_c^{\infty}(\R)$. Then under the hypothesis of Lemma \ref{apriori} and in the local chart near $\Sigma$, we have $\varphi\chi(hD_{x'})\varphi (\partial_{y}w)|_{y=0}=o_{L_{x'}^2}(1)$, where $h=\hbar^2$. Consequently, 
$$ u_2|_{\Sigma}=o_{H^{1/2}}(h^{\frac{3}{2}}),\quad \varphi\chi(hD_{x'})\varphi h\partial_yu_2|_{y=0}=o_{L^2(\Sigma)}(h).
$$	
\end{prop}
\begin{proof}
 Assume that $\varphi,\varphi_1$ are supported in a local chart and satisfy $\varphi_1|_{\text{supp}(\varphi)}=1$. In apriori, we have $\varphi\partial_{y}(\varphi_1w)|_{y=0}=\varphi\hbar^{-1}\mathcal{N}_{\hbar}((\varphi_1w)|_{y=0})=o_{H_{x'}^{-1/2}}(\hbar)$. Thus by Lemma \ref{technical1} we have
 $\varphi\hbar^{-1}\chi(\hbar^2D_{x'})\varphi\mathcal{N}_{\hbar}((\varphi_1w)|_{y=0})=o_{L_{x'}^2}(1)$.
 
\end{proof}
\section{Propagation estimate}
In this section, we will deal with the propagation estimate for $u_2$ in $H^1$, satisfying
\begin{equation}\label{eq:u2}
\begin{split}
& (h^2\Delta+1)u_2=ihf_2+h^2g_2=o_{H^1}(h^2)+o_{L^2}(h^3),\text{ in } \Omega_2,\\
&\|u_2\|_{H^1(\Omega_2)}=O(1),\; \|u_2|_{\Sigma}\|_{H^{1/2}(\Sigma)}=o(h^{3/2}),\\ &\|h\partial_{\nu}u_2|_{\Sigma}\|_{H^{-1/2}(\Sigma)}=o(h^{3/2}),\;\|\varphi\psi(hD_{x'})\varphi h\partial_{\nu} u_2|_{\Sigma}\|_{L^2(\Sigma)}=o(h).
\end{split}
\end{equation}
Set $w_2=h^{-1}u_2$. From \eqref{eq:u2}, 
$$ -\|\nabla u_2\|_{L^2(\Omega_2)}^2+\|h^{-1}u_2\|_{L^2(\Omega_2)}^2=\big\langle(\partial_{\nu}u_2)|_{\Sigma}\cdot u_2|_{\Sigma}\big\rangle^{H^{1/2}(\Sigma)}_{H^{-1/2}(\Sigma)}+o(1)=o(1).
$$
Hence we could equivalently deal with the propagation estimate for $w_2$ in $L^2$, satisfying
\begin{equation}\label{eq:w2}
\begin{split}
& (h^2\Delta+1)w_2=if_2+hg_2=o_{H^1}(h)+o_{L^2}(h^2),\text{ in } \Omega_2,\\
&\|w_2\|_{H^1(\Omega_2)}=O(h^{-1}),\|w_2\|_{L^2(\Omega_2)}=O(1),\; \|w_2|_{\Sigma}\|_{H^{1/2}(\Sigma)}=o(h^{1/2}),\\ &\|h\partial_{\nu}w_2|_{\Sigma}\|_{H^{-1/2}(\Sigma)}=o(h^{1/2}),\; \|\varphi \psi(hD_{x'})\varphi h\partial_{\nu}w_2|_{\Sigma}\|_{L^2(\Sigma)}=o(1).
\end{split}
\end{equation}
The goal of this section is to prove the invariance of the semiclassical measure $\mu$ associated with (a subsequence of) $w_2$ and finally prove that $\mu=0$ from the boundary conditions in \eqref{eq:w2} on the interface $\Sigma$. This will end the contradiction argument.
\subsection{Propagation away from \texorpdfstring{$\Sigma$}{Sigma}}
The defect measure in the interior of $\Omega_2$ for $u_2$ is defined via the following quadratic form:
$$ \phi(Q_h,w_2)=(Q_hw_2,w_2)_{L^2(\Omega_2)}:=\int_{\Omega_2}Q_hw_2\cdot \ov{w}_2dx.
$$
\begin{prop}[Interior propagation]\label{propagation:interior}
	Let $Q_h=\widetilde{\chi}Q_h\widetilde{\chi}$ be a $h$-pseudodifferential operator of order 0, where $\widetilde{\chi}\in C_c^{\infty}(\Omega_2)$, then we have
	$$ \frac{1}{ih}\big([h^2\Delta+1,Q_h] w_2, w_2 \big)_{L^2}=o(1).
	$$	
\end{prop}
\begin{proof}
By developing the commutator and using the equation \eqref{eq:u2}, we have
\begin{align*}
\big(\frac{1}{ih}[h^2\Delta+1,Q_h]w_2, w_2\big)=& \frac{1}{ih}\big( (h^2\Delta+1)Q_hw_2,w_2 \big)-\frac{1}{ih}\big( Q_h(if_2+hg_2), w_2 \big)\\
=& \frac{1}{ih}\big(Q_hw_2,if_2+hg_2 \big)+\frac{1}{ih}\big( Q_h(if_2+hg_2), w_2 \big)\\
=&o(1),
\end{align*}
where we used the integration by part without boundary terms, since the kernel of $Q_h$ is supported away from the boundary $\Sigma=\partial\Omega_2$. 
This completes the proof of Proposition \ref{propagation:interior}.
\end{proof}
\subsection{Geometry near the interface}\label{geometry}
Near $\Sigma=\partial\Omega_2$, we adopt the local coordinate system $x=(y,x')$ in $U:=(-\epsilon_0,\epsilon_0)_y\times X_{x'}$ for the tubular neighborhood of $\Sigma$, similar as in the previous section but with the new convention
$\Omega_2\cap U=(0,\epsilon_0)_y\times X_{x'}$. In this coordinate system, the Euclidean metric $dx^2$ is identified as the matrix
$$
\ov{g}=\left(\begin{matrix}
1 &0\\
0 &g(y,x')
\end{matrix}\right),\text{ or } \ov{g}^{-1}:=\left(\begin{matrix}
1 &0\\
0 &g^{-1}(y,x')
\end{matrix}
\right).
$$
  Near $\Sigma$, the defect measure $\mu$ for $w_2$ is defined via the quadratic form for tangential operators:
$$ \phi(Q_h,w_2):=\int_{U}Q_hw_2\cdot \ov{w}_2 \sqrt{|\ov{g}|}dydx',
$$
where $|\ov{g}|:=\mathrm{det}(\ov{g})$.
The principal symbol of the operator $P_{h,0}=-(h^2\Delta+1)$ is $$p(y,x',\eta,\xi')=\eta^2+|\xi'|_{\ov{g}}^2-1:=\eta^2+\langle\xi',\ov{g}^{-1}\xi'\rangle_{\R^{d-1}}-1.$$
By Char($P$) we denote the characteristic variety of $p$:
$$\textrm{Char}(P):=\{(x,\xi)\in T^*\mathbb{R}^{d}|_{\ov{\Omega}}:p(x,\xi)=0\}.
$$
Denote by $^bT\ov{\Omega}_2$ the vector bundle whose sections are the vector fields $X(p)$ on $\ov{\Omega}_2$ with $X(p)\in T_p\partial\Omega_2$ if $p\in\partial\Omega_2$. Moreover, denote by $^bT^*\ov{\Omega}_2$ the Melrose's compressed cotangent bundle which is the dual bundle of $^bT\ov{\Omega}_2$. Let 
$$ j:T^*\ov{\Omega}_2\rightarrow ^bT^*\ov{\Omega}_2
$$
be the canonical map. In our geodesic coordinate system near $\partial\Omega_2$, $^bT\ov{\Omega}_2$ is generated by the vector fields $\frac{\partial}{\partial x'_1},\cdot\cdot\cdot, 
\frac{\partial}{\partial x'_{d-1}},y\frac{\partial}{\partial y}$ and thus $j$ is defined by
$$ j(y,x';\eta,\xi')=(y,x';v=y\eta,\xi').
$$

Let $Z:=j(\textrm{Char}(P))$.
By writing in another way
$$p=\eta^2-r(y,x',\xi'),\; r(y,x',\xi')=1-|\xi'|_{\ov{g}}^2,
$$
we have the standard decomposition
$$ T^*\partial\Omega_2=\mathcal{E}\cup\mathcal{H}\cup\mathcal{G},
$$
according to the value of $r_0:=r|_{y=0}$ where
$$\mathcal{E}=\{r_0<0\},
\mathcal{H}=\{r_0>0\},
\mathcal{G}=\{r_0=0\}.
$$
The sets $\mathcal{E},\mathcal{H},
\mathcal{G}$ are called elliptic, hyperbolic and glancing, with respectively. We define also the set
$$ \mathcal{H}_{\delta}:=\{ \delta<r_0<1-\delta \}
$$
with $0<\delta<\frac{1}{2}$ for the non-tangential and non incident points.
Note that here the elliptic points $\mathcal{E}$ is different from those defined in Section \ref{sectionelliptic}. 

To classify different situations as a ray approaching the boundary, we need more accurate decomposition of the glancing set $\mathcal{G}$. Let $r_1=\partial_yr|_{y=0}$ and define
$$ \mathcal{G}^{k+3}=\{(x',\xi'):r_0(x',\xi')=0,H_{r_0}^j(r_1)=0,\forall j\leq k;H_{r_0}^{k+1}(r_1)\neq 0\}, k\geq 0\}
$$
$$
\mathcal{G}^{2,\pm}:=\{(x',\xi'):r_0(x',\xi')=0,\pm r_1(x',\xi')>0\},\mathcal{G}^2:=\mathcal{G}^{2,+}\cup\mathcal{G}^{2,-}.
$$
Next we recall the definition of the generalized bicharacteristic:
\begin{definition}\label{generalbicha} 
A generalized bicharacteristic of $\Omega_2$ is a piecewise continuous map from $\R$ to $^bT^*\ov{\Omega}_2$ such that at any discontinuity point $s_0$, the left and right limits $\gamma(s_0\mp)$ exist and are the two points above the same hyperbolic point on the boundary (this property translates the specular reflection of geometric optics) and except at these isolated points the curve is $C^1$ and satisfies
\begin{itemize}
	\item $\frac{d\gamma}{ds}(s)=H_p(\gamma(s))$ if $\gamma(s)\in T^*\Omega_2$ or $\gamma(s)\in \mathcal{G}^{2,+}$ 
	\item $\frac{d\gamma}{ds}(s)=H_p(\gamma(s))-\frac{H_p^2y}{H_y^2p}H_y $ if $\gamma(s)\in \mathcal{G}\setminus\mathcal{G}^{2,+}$ where $y$ is the boundary defining function. 
\end{itemize}	
\end{definition}
\begin{rem}
The first property in the definition above is the fact that the curve is a geodesic in the interior or passing though a non diffractive point. The second one is that passing through a non diffractive gliding point it is curved to be forced to remain in the interior of $T^*\partial\Omega_2$ for a while. When the domain is smooth and does not have infinite order of contact with its tangents, then (see \cite{MS-1}) through each point passes a unique generalized bicharacteristic. In general only existence is known.
	
\end{rem}
\begin{rem}
In the statement of the geometric control condition \ref{GCC}, the generalized rays are the projection of the generalized bicharacteristics of $\Omega$ onto $\ov{\Omega}$.  
\end{rem}

\subsection{Elliptic regularity}
\begin{lem}\label{PoissonItegral}
Denote by $\lambda(y,x',\xi')=\sqrt{|\xi'|_{\ov{g}}^2-1}$. Let $\psi\in C^{\infty}(\R^{d-1}), \varphi_1,\varphi_2\in C_c^{\infty}(\R^{d})$, such that on the support of $\psi(\xi')\varphi_1(y,x')$ and $\psi(\xi')\varphi_2(y,x')$, $|\xi'|_g>1+\delta$ for some $\delta>0$. Then we have
$$ \mathrm{Op}_h\big(\mathbf{1}_{y\geq 0}\varphi_2\mathrm{e}^{\frac{-y\lambda}{h}}\psi(\xi')\big)\varphi_1=\mathcal{O}(1): H^{-\frac{1}{2}}(\R_{x'}^{d-1})\rightarrow L^2(\R_+^d)
$$ 
\end{lem}
\begin{proof}
Denote by $T_y:=\mathrm{Op}_h\big(\mathbf{1}_{y\geq 0}\varphi_2\mathrm{e}^{\frac{-y\lambda}{h}}\psi(\xi')\big)$. By definition, we have for $f_0\in H_{x'}^{-1/2}$  and $y>0$ that
\begin{equation*}
\begin{split}
(T_yf_0)(x'):=\frac{1}{(2\pi h)^{d-1}}\iint \mathrm{e}^{-\frac{y\lambda(y,x',\xi')}{h}}\psi(\xi')\varphi_2(y,x')\mathrm{e}^{\frac{i(x'-z')\cdot\xi'}{h}}f_0(z')dz'd\xi'.
\end{split}
\end{equation*}
 Denote by $F_0:=\langle D_x'\rangle^{-1/2}f_0$, then this term can be written as
 $$ \mathrm{Op}\big(e^{-\frac{y\lambda(y,x',h\xi')}{h}}\psi(h\xi')\langle\xi'\rangle^{\frac{1}{2}}\varphi_2(y,x') \big)F_0.
 $$
For fixed $y>0$, from the Calder\'on-Vaillancourt theorem and the support property of $\psi$, we have, for any $M>0$ that
$$\big\|\mathrm{Op}\big(e^{-\frac{y\lambda(y,x',h\xi')}{h}}\psi(h\xi')\langle\xi'\rangle^{\frac{1}{2}}\varphi_2(y,x') \big)F_0\big\|_{L_{x'}^2}\leq C_Mh^{-\frac{1}{2}}e^{-\frac{cy}{h}}\big(1+\frac{y^M}{h^M}\big)\|F_0\|_{L_{x'}^2}.
$$
and the constants $C_M, c$ are independent of $y$. Squaring the inequality above and integrating in $y$ yields the bound $O(1)\|F_0\|_{L_{x'}^2}^2=O(1)\|f_0\|_{H_{x'}^{-\frac{1}{2}}}^2.$ 
This completes the proof of Lemma \ref{PoissonItegral}.
\end{proof}

\begin{prop}\label{elliptic}
	$$ \mu\mathbf{1}_{\mathcal{E}}=0.
	$$ 
\end{prop}
\begin{proof}
	Applying \eqref{pseudo-relation} to $\kappa=o_{H^1}(h)+o_{L^2}(h^2)$ and $\hbar=h$, we obtain that
	\begin{equation}\label{eqelliptic}
	\begin{split}
	\varphi(y,x')\psi(hD_{x'})\ud{w_2}=&-\mathrm{Op}_h\big(\mathrm{e}^{\frac{iy\eta_+}{h}}in_N(y,\cdot) \big)(h\partial_yw_2|_{y=0})-\mathrm{Op}_h\big(\mathrm{e}^{\frac{iy\eta_+}{h}}d_N(y,\cdot)\big)(w_2|_{y=0})\\
	+&\mathrm{Op}_h\big(\mathrm{e}^{\frac{iy\eta_+}{h}}in_N(y,\cdot) \big)(hH(0,x')w_2|_{y=0})+o_{L_{y,x'}^2}(h^2).
	\end{split}
	\end{equation}  
	Applying Lemma \ref{PoissonItegral}, we have 
	$ \mathrm{Op}_h\big(\mathrm{e}^{\frac{iy\eta_+}{h}}in_N(y,\cdot) \big)(h\partial_yw_2|_{y=0})
	=o_{L_{y,x'}^2}(h^{\frac{1}{2}}).
	$	
	By the same way, the other terms on the right side of \eqref{eqelliptic} are at most $o_{L_{y,x'}^2}(h^{\frac{1}{2}})$. Hence $\varphi(y,x')\psi(hD_{x'})\ud{w_2}=o_{L_{y,x'}^2}(h^{\frac{1}{2}})$, and this completes the proof of Proposition \ref{elliptic}. 
\end{proof}

\subsection{Propagation formula near the interface}
Consider the operator
$$ B_h=B_{0,h}+B_{1,h}h\partial_y
$$
where $B_{j,h}=\widetilde{\chi}_1\op_h(b_j)\widetilde{\chi}_1$, $j=0,1$ are two tangential operators and  $\widetilde{\chi}_1$ has compact support near a point $z_0\in \Sigma$. The symbols $b_j$ are compactly supported in $(x',\xi')$ variables. Note that in the local coordinate system,
$$ P_{h,0}=-h^2\Delta-1=-\frac{1}{\sqrt{|\ov{g}|}}h\partial_y\sqrt{|\ov{g}|}h\partial_y-R_h,
$$
where $R_h$ is a self-adjoint tangential differential operator of order $2$ classic and of order $0$ semiclassic.

\begin{lem}[Boundary propagation]\label{propagation:boundary} 
Let $(\widetilde{w}_h)$ be a $h$-dependent family of functions
satisfying $\widetilde{w}_h=O_{L^2(\Omega_2)}=O(1)$ and $\widetilde{w}_h=O_{H^{1}(\Omega_2)}(h^{-1})$. 
Assume moreover that $\widetilde{w}_h$ satisfies the equation
$$ P_{h,0}\widetilde{w}_h=o_{H^1(\Omega_2)}(h)+o_{L^2(\Omega_2)}(h^2)
$$
and the boundary condition: $\widetilde{w}_h|_{\Sigma}=o_{H^{\frac{1}{2}}}(h^{\frac{1}{2}})$ and $h\partial_{\nu}\widetilde{w}_h=O_{H^{-\frac{1}{2}}}(h^{\frac{1}{2}})$. Then we have
	\begin{equation}\label{eq:propagationinterface} \frac{1}{ih}\big([P_{h,0},B_h]\widetilde{w}_h,\widetilde{w}_h\big)_{L^2(\Omega_2)}=i\big(B_{1,h}|_{y=0}(h\partial_y\widetilde{w}_h)|_{y=0},(h\partial_y\widetilde{w}_h)|_{y=0} \big)_{L^2(\Sigma)}+o(1).
	\end{equation}
\end{lem}
\begin{proof}
First we remark that the right hand side of \eqref{eq:propagationinterface} makes sense, since $B_{1,h}|_{y=0}$ is a classical smoothing operator (but of semi-classical order $0$). 
We denote by $\widetilde{w}=\widetilde{w}_h$ for simplicity. Without loss of generality, we may assume that $B_{0,h}=0$, since the treatment for the term $\frac{1}{ih}\big([P_{h,0},B_{0,h}]\widetilde{w},\widetilde{w}\big)_{L^2}$ is the same as in the proof of Proposition \ref{propagation:interior}, which contributes only $o(1)$ terms. By expanding the commutator, we have
\begin{equation*}
\begin{split}
&\frac{1}{ih}\big([P_{0,h},B_h]\widetilde{w},\widetilde{w}\big)_{L^2}\\=&\frac{1}{ih}\big(P_{0,h}B_{1,h}h\partial_y\widetilde{w},\widetilde{w}\big)-\frac{1}{ih}\big(B_{1,h}h\partial_yP_{0,h}\widetilde{w},\widetilde{w} \big)_{L^2}\\
=&\frac{1}{ih}\big(B_{1,h}h\partial_y\widetilde{w},P_{0,h}\widetilde{w}\big)_{L^2}-\frac{1}{ih}\big(B_{1,h}h\partial_yP_{0,h}\widetilde{w},\widetilde{w}\big)_{L^2}\\
+&i\big(B_{1,h}|_{y=0}(h\partial_y\widetilde{w})|_{y=0},(h\partial_y\widetilde{w})|_{y=0}\big)_{L^2(\Sigma)}-i\big((h\partial_yB_{1,h}h\partial_y\widetilde{w})|_{y=0},\widetilde{w}|_{y=0}\big)_{L^2(\Sigma)}
\end{split}
\end{equation*}
Observe that $B_{1,h}h\partial_y\widetilde{w}=O_{L^2(\Omega_2)}(1)$, $P_{0,h}\widetilde{w}=o_{H^1(\Omega_2)}(h)+o_{L^2(\Omega_2)}(h^2)$, and $B_{1,h}h\partial_yP_{0,h}\widetilde{w}=o_{L^2(\Omega_2)}(h)+o_{H^{-1}(\Omega_2}(h^2)$, thus
$$ \frac{1}{ih}\big(B_{1,h}h\partial_yw_2,P_{0,h}\widetilde{w}\big)_{L^2}-\frac{1}{ih}\big(B_{1,h}h\partial_yP_{0,h}\widetilde{w},\widetilde{w}\big)_{L^2}=o(1)
$$
as $h\rightarrow 0$. Write
$h\partial_yB_{1,h}h\partial_y\widetilde{w}=h(\partial_yB_{1,h})h\partial_y\widetilde{w}+B_{1,h}h^2\partial_y^2\widetilde{w}$ and using the equation satisfied by $\widetilde{w},$ we obtain that
$$ h\partial_yB_{1,h}h\partial_y\widetilde{w}=A_hh\partial_y\widetilde{w}-B_{1,h}R_h\widetilde{w}-B_{1,h}P_{h,0}\widetilde{w},
$$ 
where $A_h$ is a tangential operator of order $0$ semi-classic. Thanks to Lemma \ref{technical1}, $B_{1,h}=\mathcal{O}_{L^2\rightarrow H^1}(h^{-1})$, thus
$ B_{1,h}P_{h,0}\widetilde{w}=o_{H^1(\Omega_2)}(h)
$ and by the trace theorem $(B_{1,h}P_{h,0}\widetilde{w})|_{\Sigma}=o_{H^{\frac{1}{2}}(\Sigma)}(h)$. Next since $R_h\widetilde{w}|_{\Sigma}=o_{H^{\frac{1}{2}}(\Sigma) }(h^{\frac{1}{2}})+o_{H^{-\frac{3}{2}}(\Sigma) }(h^{\frac{5}{2}})$, we have $B_{1,h}R_h\widetilde{w}|_{\Sigma}=o_{H^{\frac{1}{2}}(\Sigma)}(h^{\frac{1}{2}})$.
  We then deduce that $(h\partial_yB_{1,h}h\partial_y\widetilde{w})|_{y=0}=O_{H^{-\frac{1}{2}}(\Sigma)}(h^{\frac{1}{2}})$, which implies that $$\big((h\partial_yB_{1,h}h\partial_y\widetilde{w})|_{y=0},\widetilde{w}|_{y=0}\big)_{L^2(\Sigma)}=o(1).$$ This completes the proof of Lemma \ref{propagation:boundary}.
\end{proof}
To derive the propagation formula for the semiclassical measure, we consider a family of functions $(\widetilde{w}_h)$ satisfying the equation
$$P_{h,0}\widetilde{w}_h=o_{H^1(\Omega_2)}(h)+o_{L^2(\Omega_2)}(h^2)
$$
with a weaker boundary conditions, compared with \eqref{eq:w2}.
$$ \|\widetilde{w}_h\|_{L^2(\Omega_2)}=O(1),\;\|h\nabla \widetilde{w}_h\|_{L^2(\Omega)}=O(1),\; \|\widetilde{w}_h|_{\Sigma}\|_{H^{\frac{1}{2}}(\Sigma)}=o(h^{\frac{1}{2}}),\; \|(h\partial_{\nu}\widetilde{w}_h)|_{\Sigma}\|_{H^{-\frac{1}{2}}(\Sigma)}=O(h^{\frac{1}{2}}).
$$
Denote by $\mu$ is the semiclassical measure associated with $(\widetilde{w}_h)$.
\begin{prop}\label{muH=0}
	\begin{align}
	&(1) \quad \mu\mathbf{1}_{\mathcal{H}}=0;\\
	&(2)\quad  \limsup_{h\rightarrow 0}\big|\big(\op_h(b_0)h\partial_y \widetilde{w}_h,\widetilde{w}_h \big)_{L^2}\big|\leq \sup_{\rho\in\text{supp}(b_0)}|r(\rho)|^{\frac{1}{2}}|b_0(\rho)|, 
	\end{align}
	for any tangential symbol $b_0(y,x',\xi')$ of order 0.
\end{prop}
\begin{proof}
(1) follows from the transversality of the rays reaching $\mathcal{H}$, and the proof is the same as in \cite{BL03} (see also the proof of Proposition 2.14 in \cite{BS} by taking $M_h=0$ there). The proof of (2) is also similar as in \cite{BS}, with an additional attention when doing the integration by part. Indeed, by Cauchy-Schwartz,
$$ \big|\big(\mathrm{Op}_h(b_0)h\partial_y\widetilde{w}_h,\widetilde{w}_h\big)_{L^2}\big|\leq \big|\big(\mathrm{Op}_h(b_0)h\partial_y\widetilde{w}_h,\mathrm{Op}_h(b_0)h\partial_y\widetilde{w}_h \big)\big|^{\frac{1}{2}}\|\widetilde{w}_h\|_{L^2}.
$$
Doing the integration by part,
$$ \big(\mathrm{Op}_h(b_0)h\partial_y\widetilde{w}_h,\mathrm{Op}_h(b_0)h\partial_y\widetilde{w}_h \big)_{L^2}=O(h)-\big(\mathrm{Op}_h(b_0)h^2\partial_y^2\widetilde{w}_h,\mathrm{Op}_h(b_0)\widetilde{w}_h\big)_{L^2},
$$
where $O(h)$ comes from the commutators and the boundary term, since by the assumption on $\widetilde{w}_h$, $\big(\mathrm{Op}_h(b_0)(h\partial_y\widetilde{w}_h)|_{y=0},h\mathrm{Op}_h(b_0)(\widetilde{w}_h|_{y=0}) \big)_{L_{x'}^2}=o(h^2)$. For the rest argument, we just replace $-h^2\partial_y^2\widetilde{w}_h$ by $-R_h\widetilde{w}_h$ plus errors in $O_{L^2}(h)$. From the symbolic calculus, the contribution $\sup_{\rho}|r|^{\frac{1}{2}}|b_0(\rho)|$ comes from the principal term $\big|\big(\mathrm{Op}_h(b_0)R_h\widetilde{w}_h,\mathrm{Op}_h(b_0)\widetilde{w}_h \big)_{L^2}\big|^{\frac{1}{2}},$ after taking limsup in $h$. This completes the proof of Proposition \ref{muH=0}.
\end{proof}

\begin{lem}\label{H1measure}
	Let $B_{0,h}, B_{1,h}$ are tangential semiclassical operators of order $0$, with principal symbols $b_0, b_1$ with respectively, supported near a point $\rho_0$ of $T^*\Sigma$. Then
	\begin{equation}\label{measureidentity}
	\big(\big(B_{0,h}+B_{1,h}\frac{h}{i}\partial_y\big)\widetilde{w}_h,\widetilde{w}_h \big)= \langle\mu,b_0+b_1\eta\rangle+o(1),
	\end{equation}
	as $h\rightarrow 0$.
\end{lem}
\begin{proof}
	First we remark that the expression $\langle\mu,b_0+b_1\eta \rangle$ is well-defined, since $\mu$ belongs to the dual of $C^0(Z)$ and $\mu ( \mathcal{H})=0$, and in particular, by elliptic regularity, $\mu\mathbf{1}_{|\eta|>1}=0$.  The convergence of the quadratic form $(B_{0,h}\widetilde{w}_h,\widetilde{w}_h)$ to $\langle\mu,b_0\rangle$ is just the definition of the semiclassical measure $\mu$. If $\rho_0\in\mathcal{E}$, the contributions of both sides of \eqref{measureidentity} is $o(1)$, thanks to the elliptic regularity (see the proof of Proposition \ref{elliptic}). Next we assume that $\rho_0\in\mathcal{H}\cup\mathcal{G}$. Take $\varphi\in C_c^{\infty}(-1,1)$, $\varphi$ is equal to $1$ in a neighborhood of $(-1/2,1/2)$. For $\epsilon>0$, we write
	$$ B_{1,h,\epsilon}:=\big(1-\varphi\big(\frac{y}{\epsilon}\big)\big)B_{1,h}, \quad B_{1,h}^{\epsilon}:=B_{1,h}-B_{1,h,\epsilon}.
	$$
	Taking $h\rightarrow 0$ first  we obtain that
	$$ \big(B_{1,h,\epsilon}\frac{h}{i}\partial_y\widetilde{w}_h,\widetilde{w}_h\big)_{L^2(\Omega_2)}\rightarrow \langle\mu,\big(1-\varphi\big(\frac{y}{\epsilon}\big)\big)b_1\eta \rangle.
	$$
If $\rho_0\in\mathcal{H}$, then taking $\epsilon\rightarrow 0$, we obtain that
$$ \lim_{\epsilon\rightarrow 0}\langle\mu,\big(1-\varphi\big(\frac{y}{\epsilon}\big)\big)b_1\eta \rangle=\langle\mu,\mathbf{1}_{y>0}b_1\eta \rangle=\langle\mu, b_1\eta\rangle,
$$
since $\mu\mathbf{1}_{\mathcal{H}\cup\mathcal{E}}=0$. It remains to estimate the contribution of $\big(B_{1,h}^{\epsilon}\frac{h}{i}\partial_y\widetilde{w}_h,\widetilde{w}_h\big)$. For fixed $\epsilon>0$, we have
\begin{align*}
&\Big|\limsup_{h\rightarrow 0}\big(\big(B_{1,h}^{\epsilon}h\partial_y\widetilde{w}_h,\widetilde{w}_h\big) \big)_{L^2(\Omega_2)} \Big| \leq  \limsup_{h\rightarrow 0}\Big(\|\varphi(y/\epsilon)B_{1,h}^*\widetilde{w}_h\|_{L^2(\Omega_2)}\|h\partial_y\widetilde{w}_h\|_{L^2(\Omega_2)} \Big).
\end{align*}
Since on supp$(\mu\mathbf{1}_{y>0})$, $|\eta|\leq 1$, together with the fact that $\rho_0\in\mathcal{H}\cap\mathcal{E}$, we deduce that the right side converges to $0$ as $\epsilon\rightarrow 0$.

Now suppose that $\rho_0\in\mathcal{G}$. For any $\epsilon>0, \delta>0$, we decompose $B_{1,h}=B_{1,h}^{\epsilon}+B_{1,h}^{\epsilon,\delta}+B_{1,h,\delta}^{\epsilon}$, with
\begin{align*}
& B_{1,h,\epsilon}=\Big(1-\varphi\big(\frac{y}{\epsilon}\big)\Big)B_{1,h},\\
& B_{1,h}^{\epsilon,\delta}=\mathrm{Op}_h\Big(\varphi\big(\frac{y}{\epsilon}\big)\varphi\big(\frac{r}{\delta}\big)\Big)B_{1,h},\\
&B_{1,h,\delta}^{\epsilon}=\mathrm{Op}_h\Big(\varphi\big(\frac{y}{\epsilon}\big)\Big(1-\varphi\big(\frac{r}{\delta}\big)\Big)\Big)B_{1,h}.
\end{align*} 
By the same argument, we have
$$ \lim_{\epsilon\rightarrow 0}\lim_{h\rightarrow 0}\big(B_{1,h,\epsilon}hD_y\widetilde{w}_h,\widetilde{w}_h\big)_{L^2(\Omega_2)}=\langle\mu,b_1\eta\mathbf{1}_{y>0}\rangle=\langle\mu,b_1\eta\mathbf{1}_{\rho\notin\mathcal{H}}\rangle,
$$
since $\mu\mathbf{1}_{\mathcal{H}\cup\mathcal{E}}=0$. Next, from (2) of Proposition \ref{muH=0}, we have
$$ \limsup_{\epsilon\rightarrow 0}\limsup_{h\rightarrow 0}\big(B_{1,h}^{\epsilon,\delta}hD_y\widetilde{w}_h,\widetilde{w}_h\big)_{L^2(\Omega_2)}\leq C\delta,
$$
which converges to $0$ if we let $\delta\rightarrow 0$. Finally, by Cauchy-Schwartz, 
$$\big|\big(B_{1,h,\delta}^{\epsilon}hD_y\widetilde{w}_h,\widetilde{w}_h \big)_{L^2(\Omega_2)}\big|\leq \|hD_y\widetilde{w}_h\|_{L^2(\Omega_2)}\Big\|B_{1,h}^*\mathrm{Op}_h\Big(\varphi\big(\frac{y}{\epsilon}\big)\Big(1-\varphi\big(\frac{r}{\delta}\big)\Big) \Big)^*\widetilde{w}_h \Big\|_{L^2}.
$$
Taking the triple limit, we have
$$ \limsup_{\delta\rightarrow 0}\limsup_{\epsilon\rightarrow 0}\limsup_{h\rightarrow 0}\big|\big(B_{1,h,\delta}^{\epsilon}hD_y\widetilde{w}_h,\widetilde{w}_h\big)_{L^2(\Omega_2)}\big|\leq \langle\mu,|b_1|^2\mathbf{1}_{y=0}\mathbf{1}_{r\neq 0}\rangle=0,
$$
since $\mu\mathbf{1}_{\mathcal{E}\cup\mathcal{H}}=0$. This completes the proof of Lemma \ref{H1measure}. 
\end{proof}

As in \cite{BL03}, we define the function
\begin{align*}
\theta(y,x';\eta,\xi')=\frac{\eta}{|\xi'|} \text{ if } y>0; \quad \theta(y,x',\eta,\xi')=i\frac{\sqrt{-r_0(x',\xi')}}{|\xi'|} \text{ on } \mathcal{E}.
\end{align*} 
Since $\mu\mathbf{1}_{\mathcal{H}}=0$, $\theta$ is $\mu$ almost everywhere defined as a function on $Z$. Formally, 
$$ \sigma\big(\frac{i}{h}[P_{h,0},B_h]\big)=\{\eta^2-r,b_0+b_1\eta \}=a_0+a_1\eta+a_2\eta^2,
$$
where
\begin{align}\label{relation}
a_0=b_1\partial_yr-\{r,b_0\}',\quad a_1=2\partial_yb_0-\{r,b_1\}',\quad a_2=2\partial_yb_1,
\end{align}
and $\{\cdot,\cdot\}'$ is the Poisson bracket for $(x',\xi')$ variables. By expanding the commutator, we find
\begin{align}\label{Malagrange}
\frac{i}{h}[P_{h,0},B_h]=A_0+A_1hD_y+A_2h^2D_y^2+h\op_h(S_{\partial}^{0}+S_{\partial}^0\eta),
\end{align} 
where $A_0,A_1,A_2$ are tangential operators with symbols $a_0,a_1,a_2$, with respectively, and $S_{\partial}^0$ stands for the tangential symbol class of order $0$. 
We now have all the ingredients to present the propagation formula for the defect measure in the spirit of \cite{BL03}:
\begin{prop}\label{propagationformula}
	Assume that $B_h=B_{h,0}+B_{h,1}hD_y$, where $B_{h,0}, B_{h,1}$ are tangential operators of order 0 with symbols $b_0, b_1$, with respectively. Assume that $b=b_0+b_1\eta$. Define the formal Poisson bracket
	$$ \{p,b\}=(a_0+a_2r)+a_1\theta|\xi'|\mathbf{1}_{\rho\notin\mathcal{H}},
	$$	
	where $a_0,a_1,a_2$ are given by \eqref{relation}.
	Then any defect measures $\mu,\nu_0$ of $(\widetilde{w}_h)$, $(h\partial_{\nu}\widetilde{w}_h)|_{\Sigma}$ satisfy the relation
	$$ \langle\mu,\{p,b\}\rangle=-\langle\nu_0,b_1 \rangle.
	$$
	Moreover, if $b\in C^0(Z)$, we have
	
\begin{equation}\label{propag-1}
 \langle\mu,\{p,b\} \rangle=0.
	\end{equation}
\end{prop}
\begin{proof}
See \cite{BL03}.
\end{proof}
Moreover, we have
\begin{prop}\label{zerodiffractive}
	$\mu\big(\mathcal{G}^{2,+}\big)=0$.
\end{prop}

As showed in \cite{BL03}, we obtain that the measure $\mu$ is invariant by the flow of Melrose-Sj\"ostrand. More precisely, we have
\begin{thm}[\cite{BL03}]\label{propagationtheorem} 
Assume that $\mu$ is a semi-classical measure on $^bT^*\ov{\Omega}$ associated with the sequence $(\widetilde{w}_h)$ satisfying \eqref{propag-1} and Proposition~\ref{zerodiffractive}. Then $\mu$ is invariant under the Melrose-Sj\"ostrand flow $\phi_s$.
\end{thm}

\begin{rem}
This is a consequence of Theorem 1 in \cite{BL03} which asserts the equivalence between the measure invariance and the propagation formula $H_p(\mu)=0$ together with $\mu(\mathcal{G}^{2,+})=0$. Though Theorem 1 in \cite{BL03} is stated and proved in the context of micro-local defect measure, it also holds true in the context of the semiclassical measure from the word-by-word translation.
\end{rem}

\subsection{The last step to the proof of the resolvent estimate in Theorem \ref{thm:resolvent}}
In this subsection, we take $\widetilde{w}_h=w_2$. To finish the contradiction argument in the proof of \eqref{eq:resolvent}, it suffices to show that $\mu=0$. Let $\mu$ be the corresponding semiclassical measure and $\nu_0$ be the semiclassical measure of $h\partial_{\nu}u_2|_{\Sigma}$. Since $h\partial_{\nu}u_2|_{\Sigma}=o_{H^{-\frac{1}{2}}(\Sigma}(h^{\frac{1}{2}})$, we have that $\langle\nu_0,b_1\rangle=0$ for any compactly supported symbol $b_1(x',\xi')$. Thanks to (H), along the Melrose-Sj\"ostrand flow of $^bT^*\ov{\Omega}_2$ issued from points in $^bT^*\ov{\Omega}_2$, there must be some points reaching $\mathcal{H}(\Sigma)\cup\mathcal{G}^{2,-}(\Sigma)$. By the property of the Melrose-Sj\"ostrand flow on $^bT^*\ov{\Omega}_2$, to show that $\mu=0$, we need to verify that
$$ \mu\big(\mathcal{G}^{2,-}(\Sigma)\big)=0
$$ 
and $\mu=0$ near a neighborhood of $\rho_0\in\mathcal{H}(\Sigma)$.
\begin{prop}\label{zerogliding}
$\mu\big(\mathcal{G}^{2,-}(\Sigma)\big)=0.$	
\end{prop}
\begin{proof}
The proof is exactly the same as the proof of Proposition \ref{zerodiffractive}.	We will make use of the formula
	$$ \langle\mu,\{p,b\}\rangle=\langle\nu_0,b_1\rangle
	$$	
	by choosing $b=b_{1,\epsilon}\eta$ with
	$$ b_{1,\epsilon}(y,x',\xi')=\psi\big(\frac{y}{\epsilon^{\frac{1}{2}}}\big)\psi\big(\frac{r(y,x',\xi')}{\epsilon}\big)\kappa(y,x',\xi'),
	$$ 
	where $\psi\in C_c^{\infty}(\R)$ equals to $1$ near the origin and $\kappa(y,x',\xi')\geq 0$ near a point $\rho_0\in\mathcal{G}^{2,-}$. Note that $\{p,b_{\epsilon} \}=(a_{0}+a_2r)+a_1\eta\mathbf{1}_{\rho\notin\mathcal{H}}$, and $a_0,a_1,a_2$ are given by the relation \eqref{relation}. In particular for our choice, by direct calculation we have
	$$ a_0=b_{1,\epsilon}\partial_y r, \quad a_1=-\{r,\kappa\}'\psi\big(\frac{y}{\epsilon^{\frac{1}{2}}}\big)\psi\big(\frac{r}{\epsilon}\big),$$
	and
	$$ a_2=2\partial_yb_{1,\epsilon}=2\epsilon^{-\frac{1}{2}}\psi'\big(\frac{y}{\epsilon^{\frac{1}{2}}}\big)\psi\big(\frac{r}{\epsilon}\big)\kappa+2\frac{\partial_yr}{\epsilon}\psi\big(\frac{y}{\epsilon^{\frac{1}{2}}}\big)\psi'\big(\frac{r}{\epsilon}\big)\kappa+2\psi\big(\frac{y}{\epsilon^{\frac{1}{2}}}\big)\psi\big(\frac{r}{\epsilon}\big)\partial_y\kappa.
	$$
	Observe that $a_2$ is uniformly bounded in $\epsilon$ and for any fixed $(y,x',\xi')$, $ra_2\rightarrow 0$ as $\epsilon\rightarrow 0$. Thus from the dominating convergence, we have
	$$ \lim_{\epsilon\rightarrow 0}\langle\mu,\{p,b_{\epsilon} \}\rangle=\langle\mu,\kappa|_{y=0}\partial_yr\mathbf{1}_{r=0}\rangle.
	$$
	Since $\partial_yr<0$ on $\mathcal{G}^{2,-}$, while $-\langle\nu_0,b_{\epsilon}\rangle=0$, we deduce that $\mu\mathbf{1}_{\mathcal{G}^{2,-}}=0$. This completes the proof of Proposition \ref{zerogliding}.
\end{proof}
\begin{prop}\label{zerohyperbolic}
Let $\rho_0\in\mathcal{H}(\Sigma)$. Let $b(y,x',\xi')$ be a tangential symbol, supported near $\rho_0$. Then
$$ \|\mathrm{Op}_h(b)w_2\|_{L^2(\Omega_2)}+\|h\partial_y\mathrm{Op}_h(b)w_2\|_{L^2(\Omega_2)}=o(1),
$$
as $h\rightarrow 0$.
\end{prop}
\begin{proof}
 Since the Melrose-Sj\"ostrand flow is transverse to $\mathcal{H}$, by localizing the symbol $b$, it suffices to prove the same estimate by replacing $b$ to $q^{\pm}$, where $q^{\pm}$ is the solutions of
$$ \partial_yq^{\pm}\mp H_{\sqrt{r}(y,x',\xi')}q^{\pm}=0,\quad q^{\pm}|_{y=0}=q_0,
$$ 
and $q_0$ is supported in a sufficiently small neighborhood of $\rho_0$. Near $\rho_0$, it follows from \cite{BL03} that we can factorize $P_{h,0}$ as $\big(hD_y-\Lambda_h^+(y,x',hD_{x'})\big)\big(hD_y+\Lambda_h^-(y,x',hD_{x'}) \big)+O_{H^{\infty}}(h^{\infty})$, and also
$ \big(hD_y-\widetilde{\Lambda}_h^+(y,x',hD_{x'})\big)\big(hD_y+\widetilde{\Lambda}_h^-(y,x',hD_{x'}) \big)+O_{H^{\infty}}(h^{\infty}),
$
 where $\Lambda_h^{\pm}$ and $\widetilde{\Lambda}_h^{\pm}$ have principal symbols $\pm\sqrt{r(y,x',\xi')}$. Denote by $Q_h^{\pm}=\mathrm{Op}_h(q^{\pm})$ and set
$$ w_2^{+}:=\varphi(y)Q_h^{+}(hD_y-\Lambda_h^{-})w_2,\quad w_2^-:=\varphi(y)Q_h^{-}(hD_y-\widetilde{\Lambda}_h^+)w_2,
$$
where the cutoff $\varphi(y)$ is supported on $0\leq y\leq \epsilon_0\ll 1$ and is equal to $1$ for $0\leq y\leq \epsilon_0/2$. From the equation of $w_2$, we have
\begin{align*}
 (hD_y-\Lambda_h^+)w_2^{+}=&\varphi(y)[hD_y-\Lambda_h^+,Q_h^+](hD_y-\Lambda_h^-)w_2-ih\varphi'(y)Q_h^+(hD_y-\Lambda_h^-)w_2+o_{L_{y,x'}^2}(h)\\
 =&-ih\varphi'(y)Q_h^+(hD_y-\Lambda_h^-)w_2+o_{L_{y,x'}^2}(h),
\end{align*}
since the principal symbol of $\frac{1}{ih}[hD_y-\Lambda_h^+,Q_h^+]$ is zero, thanks to the choice of symbols $q^{\pm}$. Multiplying by $\ov{w}_2^+$ to both sides and integrating, we have for $y\leq \epsilon_0/2$ (thus $\varphi'(y)=0$) that
\begin{align}\label{energyestimate}
h\|w_2^+(y,\cdot)\|_{L_{x'}^2}^2\leq h\|w_2^{+}(0,\cdot)\|_{L_{x'}^2}^2+o(h).
\end{align}
Since $\mathrm{Op}_h(q_0)(h\partial_yw_2)|_{y=0}=o_{L_{x'}^2}(1)$, we deduce by definition that $w_2^+(0)=o_{L_{x'}^2}(1)$. This together with \eqref{energyestimate} yields $w_2^+(y)=o_{L_{x'}^2}(1)$, uniformly for all $0\leq y\leq \epsilon_0/2$. Thus $w_2^+=o_{L_{y,x'}^2}(1)$. Similar argument for $w_2^-$ yields $w_2^-=o_{L_{y,x'}^2}(1)$. Note that $hD_y-\Lambda_h^-$ is elliptic on the support of $q^+$, we deduce that $Q_h^+w_2=o_{L_{y,x'}^2}(1)$. This means that $\mu$ is zero near the support of $q^+$, hence the proof of Proposition \ref{zerohyperbolic} is complete.
\end{proof}
Consequently, we have shown that the measure $\mu$ is invariant along the bicharacteristic flow on $\Omega _2$, it vanishes near every hyperbolic point of $\Sigma$,  and $\mu( \mathcal{G}^{2,-}) =0$.  Thus $\mu$  is supported on bicharacteristics which encounter $\Sigma$ only at points of 
$$\mathcal{G}^{2<}= \cup_{k\geq 3 }\mathcal{G}^k.$$
These bicharacteristics are consequently near $\Sigma$ integral curves of $H_p$ (because in Definition~\ref{generalbicha}, the two vector fields $H_p$ and $H_p - \frac{ H_p^2(y}{H^2 _y p } H_y$ coincide on $\mathcal{G}^{2<}$).  However, according to the geometric condition assumption, all such bicharacteristics must leave $\Omega_2$. As a consequence, $\mu$ is supported on the emptyset, and hence $\mu =0$. This gives a contradiction. The proof of \eqref{eq:resolvent} in Theorem~\ref{thm:resolvent} is now complete.

\section{Optimality of the resolvent estimate}
In this section we prove the second part of Theorem \ref{thm:resolvent}. For simplicity, we consider the model case $\Omega_2=\mathbb{D}:=\{x\in\R^2:|x|<1 \}$ and $a(x)=\mathbf{1}_{\Omega_1}$ and $\Sigma=\mathbb{S}^1$.
To prove the second part in Theorem~\ref{thm:resolvent} we need to construct functions $u_1,v_1,u_2,v_2$, such that
$$ \|(u_j,v_j)\|_{H^1\times L^2(\Omega_j)}\sim 1,\; \|(f_j,g_j)\|_{H^1\times L^2(\Omega_j)}=O(h), \quad j=1,2
$$
\begin{align*}
\begin{cases} 
& u_1=ih(f_1-v_1), \text{ in }\Omega_1\\
& h\Delta u_1+h\Delta v_1-i v_1=hg_1,\text{ in } \Omega_1 \\
& u_2=ih(f_2-v_2), \text{ in } \Omega_2\\
& h\Delta u_2-iv_2=hg_2,\text{ in }\Omega_2
\end{cases}
\end{align*}
together with the boundary condition on the interface
\begin{align*}
u_1|_{\Sigma}=u_2|_{\Sigma},\quad \partial_{\nu}u_2|_{\Sigma}=(\partial_{\nu}u_1+\partial_{\nu}v_1 )|_{\Sigma},
\end{align*}
The key point in the construction is that  in $\Omega_1$, we construct quasi-modes concentrated at the scale $|D_x|\sim \hbar^{-1}=h^{-\frac{1}{2}}$ while in $\Omega_2$, the quasi-modes are concentrated at the scale $|D_{x'}|\sim |D_y|\sim |D_x|\sim h^{-1}$ near the interface $\Sigma$, where $x'$ is the tangential variable near $\Sigma$ and $y$ is the normal variable. 
Now we describe the construction.

\noi
$\bullet$ {\bf  Step 1: Construction at the zero order:}
We first choose $u_2^{(0)}$, such that
$$ h^2\Delta u_2^{(0)}+u_2^{(0)}=0, \quad u_2^{(0)}|_{\Sigma}=0;\quad \|\nabla u_2^{(0)}\|_{L^2(\Omega_2)}\sim h^{-1}\|u_2^{(0)}\|_{L^2(\Omega_2)}\sim 1.
$$
Moreover, we require that $u_2^{(0)}$ such that they are  hyperbolically localized, in the sense that
\begin{align}\label{localizationstep1}
\|\partial_{\nu}u_2^{(0)}|_{\Sigma}\|_{H^s(\Sigma)}\sim h^{-s} ,
\quad
 \mathrm{WF}_h(\partial_{
\nu}u_2^{(0)}|_{\Sigma})\subset \mathcal{H}_{\delta}(\Sigma):=\{(x',\xi'):\delta<r_0(x',\xi')<1-\delta \}
\end{align}
for some $0<\delta<\frac{1}{2}$.
The existence of such sequence of eigenfunctions is not difficult to prove in the case of a disc or an ellipse, we postpone this fact in Lemma \ref{energynorm} of the Appendix. This will actually be the only point where in Theorem~\ref{thm:decay} we use the particular choice $\Omega_2 = \mathbb{D}$.

 Next we define
$$ v_2^{(0)}=ih^{-1}u_2^{(0)},\quad f_2^{(0)}=g_2^{(0)}=0.
$$
From \eqref{localizationstep1}, we have
\begin{equation}
 \partial_{\nu}u_2^{(0)}|_{\Sigma}= \begin{cases} O_{L^2(\Sigma)}(1)\\ O_{H^{-\frac{1}{2}}(\Sigma)}(h^{\frac{1}{2}})\\ O_{H^{\frac{1}{2}}(\Sigma)}(h^{-\frac{1}{2}}).\end{cases}
 \end{equation}
 
We remark that here we use the fact that the dimension $d\geq 2$.

Next we solve the elliptic equation with the mixed Dirichlet-Neumann data (with $\hbar=h^{\frac{1}{2}}$):
$$ (\hbar^2\Delta-i)w^{(0)}=0,\quad \partial_{\nu}w^{(0)}|_{\Sigma}=\partial_{\nu}u_2^{(0)}|_{\Sigma},\quad w^{(0)}|_{\partial\Omega_1\setminus\Sigma}=0.
$$
From Proposition \ref{pb:mixed}, there exists a unique solution $w^{(0)}$ of this system, which satisfies

\begin{equation}
 w^{(0)}=\begin{cases}
 O_{H^2(\Omega_1)}(\hbar^{-1})\\
 O_{H^1(\Omega_1)}(\hbar) \\
  O_{L^2(\Omega_1)}(\hbar^2),
 \end{cases}
  \end{equation}

and hence by interpolation 
$$ w^{(0)}=O_{H^{\frac 3 2}(\Omega_1)}(1)
$$
and by trace theorems
\begin{equation}
  w^{(0)}\mid_{\Sigma} = \begin{cases}
   O_{H^{\frac 1 2} ( \Sigma)} (\hbar)\\
    O_{H^{1} ( \Sigma)} (1)
    \end{cases}
    \end{equation}
    
 Moreover, from the information of $\mathrm{WF}_h(\partial_{\nu}u_2^{(0)}|_{\Sigma})$ and Proposition \ref{control:frontdonde}, we have 
 $$\mathrm{WF}_h(w^{(0)}|_{\Sigma})\subset \mathrm{WF}_h(\partial_{\nu}w^0|_{\Sigma})\subset \mathcal{H}_{\delta}(\Sigma).$$ Hence
$$   \| w^{(0)}\mid_{\Sigma}\|_{H^1( \Sigma)} \sim h^{- \frac 1 2}  \| w^{(0)}\mid_{\Sigma}\|_{H^{\frac 1 2}( \Sigma)} = O(1)
$$
Next we define $u_1^{(0)}, v_1^{(0)}$ such that
$$ v_1^{(0)}=ih^{-1}u_1^{(0)}, \quad w^{(0)}=u_1^{(0)}+v_1^{(0)} = (1 + i h^{-1}) u_1^{(0)};\quad f_1^{(0)}=0,\quad g_1^{(0)}=ih^{-1}u_1^{(0)}= v^{(0)}_1.
$$
This implies
\begin{equation}
u_1^{(0)}=\begin{cases}
O_{H^1(\Omega_1)}(h^{\frac{3}{2}})\\
O_{L^2(\Omega_1)}(h^{2}),
\end{cases}
\end{equation}
 and consequently $g_1^{(0)}=O_{L^2(\Omega_1)}(h)$. 

In summary, as the first step, we have constructed   quasi-modes $(u_1^{(0)},v_1^{(0)};u_2^{(0)},v_2^{(0)})$ and $(f_1^{(0)}=0,g_1^{(0)};f_2^{(0)}=0,g_2^{(0)}=0)$ such that
 \begin{equation}\label{0approximation}
\begin{cases}
& h\Delta(u_1^{(0)}+v_1^{(0)})-iv_1^{(0)}=hg_1^{(0)}, \qquad g_1 = O_{L^2( \Omega_1)}(h)\\
&u_1^{(0)}=-ihv_1^{(0)}\\
& h\Delta u_2^{(0)}-iv_2^{(0)}=0\\
& u_2^{(0)}=-ih v_2^{(0)}\\
&\partial_{\nu}u_2^{(0)}|_{\Sigma}=\big(\partial_{\nu}u_1^{(0)}+\partial_{\nu}v_1^{(0)} \big)|_{\Sigma},\\
& \big(u_2^{(0)}-u_1^{(0)}\big)|_{\Sigma}=O_{H^{\frac{1}{2}}(\Sigma)}(h^{\frac{3}{2}}),\quad \big(v_2^{(0)}-v_1^{(0)}\big)|_{\Sigma}=O_{H^{\frac{1}{2}}(\Sigma)}(h^{\frac{1}{2}}),
\end{cases}
\end{equation}
and to conclude the proof of Theorem~\ref{thm:resolvent}, it remains to eliminate the error term in the last boundary condition in~\eqref{0approximation}. An important point is that both $u^{(0)}_2$ and $u^{(0)}_1$ (and hence also $uv{(0)}_2$ and $v^{(0)}_1$) have their wave front included in $\mathcal{H}_\delta ( \Sigma)$.

\noi
$\bullet$ {\bf  Step 2: Construction at the first order:}
We now introduce correction terms to eliminate the error term in the last boundary condition of \eqref{0approximation}.
We are looking for a correction term $e_2 ^{(1)}$, 
$$ u_2^{(1)}=u_2^{(0)}+e_2^{(1)},\quad v_2^{(1)}=i h^{-1} u_2^{(1)}=v_2^{(0)}+i h^{-1} e_2^{(1)}, 
$$
while keeping all other terms identical
$$ u_1^{(1)}= u_1^{(0)}, \quad v_1^{(1)}= v_1^{(0)},
$$
First,  using the geometric optics construction (see Appendix), we construct $\widetilde{e}_2^{(1)}$, solving  near $\Sigma$, solving for $N$ large enough to be fixed later 
$$ (h^2\Delta+1)\widetilde{e}_2^{(1)}=O_{L^2}(h^N)$$
near $\Sigma$, and the boundary conditions
\begin{equation}\label{boundary-cond}
 \widetilde{e}_2^{(1)}|_{\Sigma}=(u_1^{(0)}-u_2^{(0)})|_{\Sigma}+O_{L^2(\Sigma)}(h^N), \qquad \partial_\nu \widetilde{e}^{(1)}_2 \mid _{\Sigma} = O_{L^2( \Sigma)} ( h^N).
\end{equation}
 with $h$-semiclassical wave front sets of all the functions are localized near $\mathcal{H}_\delta (\Sigma)$. 

\begin{align}\label{go}
 (h^2\Delta+1)\widetilde{e}_2^{(1)}=O_{L^2(\Omega_2)}(h^N),\quad \widetilde{e}_2^{(1)}=(u_1^{(0)}-u_2^{(0)})|_{\Sigma}+O_{H^N(\Sigma)}(h^N),\quad h\partial_{\nu}\widetilde{e}_2^{(1)}|_{\Sigma}=O_{H^N(\Sigma)}(h^N)
\end{align}
locally near $x_0\in\Sigma$. We then take a cutoff $\chi$, such that $\chi\equiv 1$  on  $\Sigma$, and with support sufficiently close to $\Sigma$ so that $\widetilde{e}$ is defined on the support of $\chi$  (i.e.  $\chi$ vanishes  along the bicharacteristics, before the formation of the  caustics). Let $e_2^{(1)}:=\chi \widetilde{e}_2^{(1)}.$ Hence
\begin{equation}
\begin{gathered}
 (h^2\Delta+1)e_2^{(1)}=[h^2\Delta,\chi]\widetilde{e}_2^{(1)}+O_{L^2(\Omega_2)}(h^4)=O_{L^2(\Omega_2)}(h^3),\\
  e_2^{(1)}|_{\Sigma}=\widetilde{e}_2^{(1)}|_{\Sigma},\quad h\partial_{\nu}e_2^{(1)}|_{\Sigma}=h\partial_{\nu}\widetilde{e}_2^{(1)}|_{\Sigma}.
  \end{gathered}
  \end{equation}
 
Again, all the functions and the errors are microlocalized near $(x_0,\xi_0)\in \mathcal{H}_\delta(\Sigma)$. Moreover, from the boundary conditions~\eqref{boundary-cond} which determine the values of the symbols $b^{\pm}$ in the geometric optics construction, we have 
$$ \|e_2^{(1)}\|_{H^1(\Omega_2)}=O(h),\quad \|e_2^{(1)}\|_{L^2(\Omega_2)}=O(h^2),\quad \partial_{\nu}e_2^{(1)}|_{\Sigma}=O_{L^2(\Sigma)}(h^{N-1}).
$$
 The geometric optics constructions in the appendix are local, but using a partition of unity of $\Sigma$, we choose a finite cutoff functions $(\chi_j)_{j=1}^M$ to replace $\chi$ and modify the function $e_2^{(1)}$ by
$$  e_2^{(1)}:=\sum_{j=1}^M \chi_j\widetilde{e}_{2,j}^{(1)},
$$
where $\widetilde{e}_{2,j}^{(1)}$ is the corresponding geometric optics near supp$(\chi_j)$.

Next we define $g_2^{(1)}=h^{-2}\cdot(h^2\Delta+1)e_2^{(1)}$.
We now have 
 \begin{equation}\label{1approximation}
\begin{cases}
& h\Delta(u_1^{(1)}+v_1^{(1)})-iv_1^{(1)}=hg_1^{(1)}, \qquad g_1 = O_{L^2( \Omega_1)}(h)\\
&u_1^{(1)}+ihv_1^{(1)}= 0\\
& h\Delta u_2^{(1)}-iv_2^{(1)}=h g_2^{(1)}, \qquad g_2^{(1)} = O_{L^2( \Omega_1)} (h)\\
& u_2^{(1)}=-ih v_2^{(1)}\\
&\partial_{\nu}u_2^{(1)}|_{\Sigma}=\big(\partial_{\nu}u_1^{(1)}+\partial_{\nu}v_1^{(1)} \big)|_{\Sigma} +O_{H^{N}(\Sigma)}(h^N)\\
& \big(u_2^{(1)}-u_1^{(1)}\big)|_{\Sigma}=O_{H^N(\Sigma)}(h^N),\quad \big(v_2^{(1)}-v_1^{(1)}\big)|_{\Sigma}=O_{H^{N}(\Sigma)}(h^{N-1}),
\end{cases}
\end{equation}
It now remains to eliminate completely the errors in the last boundary condition in~\eqref{1approximation}. For this we just use  the trace operators. Recall that if $s> \frac 3 2$, the map 
$$ \Gamma: u \in H^s( \Omega_1) \mapsto ( u \mid_\Sigma, \partial_\nu u \mid_{\Sigma}) \in H^{s- 1/2}( \Sigma)\times H^{s- 3/2}(\Sigma)
$$
is continuous surjective and admits a bounded right inverse.
As a consequence, if $N$ is large enough,  there exists $e_2^{(2)}\in H^{N- \frac 3 2}$ (supported near $\Sigma$) such that 
$$ \|e_2^{(2)} \|_{ H^{N- \frac 3 2}(\Omega_2)} = O(h^N), \quad e_2^{(2)}\mid_\Sigma = (u_1^{(1)}-u_2^{(1)}) \mid_{\Sigma}, \quad \partial_{\nu}e_2^{(2)}|_{\Sigma}=\big(\partial_{\nu}u_1^{(1)}-\partial_{\nu}v_1^{(1)}\big)\mid_{\Sigma}  - \partial_\nu u_2^{(1)} \mid_{\Sigma} 
$$

Choosing now 
$$ u_2^{(2)}=u_2^{(1)}+ e_2^{(2)}, \qquad v_2^{(2)}=v_2^{(1)}+ i h^{-1} e_2^{(2)}, \qquad  g_2^{(2)} = g_2^{(1)} +h^{-1} ( h^2 \Delta + 1) e_2^{(2)}
$$
and keeping the other terms identical
$$ u_1^{(2)}= u_1^{(0)}, \quad v_1^{(2)}= v_1^{(0)},\quad g_1^{(2)} = g_1^{(1)},
$$
we get (if $N$ is large enough)
\begin{equation}\label{2approximation}
\begin{cases}
& h\Delta(u_1^{(2)}+v_1^{(2)})-iv_1^{(2)}=hg_1^{(2)}, \qquad g_1 = O_{L^2( \Omega_1)}(h)\\
&u_1^{(2)}+ihv_1^{(2)}=0\\
& h\Delta u_2^{(2)}-iv_2^{(2)}= h g_2^{(2)}, \qquad g_2^{(2)} = O_{L^2( \Omega_1)} (h)\\
& u_2^{(2)}=-ih v_2^{(2)}\\
&\partial_{\nu}u_2^{(2)}|_{\Sigma}=\big(\partial_{\nu}u_1^{(2)}+\partial_{\nu}v_1^{(2)} \big)|_{\Sigma} \\
& \big(u_2^{(2)}-u_1^{(2)}\big)|_{\Sigma}=0,\quad \big(v_2^{(2)}-v_1^{(2)}\big)|_{\Sigma}=0
\end{cases}
\end{equation}
This ends the proof of the construction of quasi-modes in Theorem~\ref{thm:resolvent}. \qed 

\section{Appendix: Technical ingredients }

\subsection{Elliptic problem with mixed Dirichlet Neumann data}

Let $U\subset \R^d$ be a bounded domain with smooth boundary. For $F\in C^{\infty}(\ov{U})$, we denote by
$$ \gamma^0(F)=F|_{\partial U},\; \gamma^1(F)=(\partial_{\nu}F)|_{\partial U}
$$
the Dirichlet and Neumann trace, with respectively. From the trace theorem, we know that
$$ \gamma^0: H^s(U)\rightarrow H^{s-\frac{1}{2}}(U)
$$
is bounded and surjective. 
Let 
$$ \mathcal{H}^1_{0} ( \Omega_1) = \{ v \in H^1( \Omega); v\mid_{\partial\Omega_1 \setminus \Sigma} =0\},
$$

We prove the following existence result of the mixed Dirichlet-Neumann boundary value problem:
\begin{prop}\label{pb:mixed}
	For any $F \in H^{-\frac 1 2} ( \Sigma)$, the boundary value problem (note that $\partial\Omega_1=\Sigma\cup\partial\Omega $ and $\Sigma,\partial\Omega$ are separated)
	
	\begin{align}
	 (\hbar ^2 \Delta -i) w&=0, \label{syst-1}\\
	  \partial_\nu w \mid_\Sigma = F, \qquad & w \mid_{\partial \Omega_1 \setminus \Sigma} =0 \label{syst-2}
	  \end{align}
	
	 admits a unique solution $w\in \mathcal{H}^1_0(\Omega_1)$ satisfying 
	$$ \Bigl( \hbar \| \nabla_x w\|_{L^2( \Omega_1)}  +  \|  w\|_{L^2( \Omega_1)} \Bigr) \leq C\hbar \| F\|_{H^{-\frac 1 2} ( \Sigma)}.	$$ 
	Furthermore,  if $F\in H^{\frac 1 2} ( \Sigma)$, then  $w\in {H}^2(\Omega_1)$ and
	$$  \| \nabla^2_x w\|_{L^2( \Omega_1)}   \leq C\Bigl( \| F\|_{H^{\frac 1 2} ( \Sigma)} + \hbar^{-1}\| F\|_{H^{- \frac 1 2} ( \Omega_1)}\Bigr).	$$ 
	
\end{prop}

\begin{proof}
We just sketch the proof which is a variation around very classical ideas.  Multiplying~\eqref{syst-1} by $\overline{\varphi}$ vanishing on $\partial\Omega_1 \setminus \Sigma$ and integrating by parts using Greens formula, we get 
$$ 0 = \int_{\Omega_1 } (\hbar ^2 \Delta -i) w\overline{\varphi} (x) dx =  \int_{\Omega_1 } -\hbar ^2  \nabla_x w \nabla_x \overline{\varphi} -iw \overline{\varphi} (x) dx + \int_\Sigma \hbar^2 \partial_\nu w \overline{\varphi} (x) d \sigma
$$ 
As a consequence, if 
 the function $w$ satisifes~\eqref{syst-1} ~\eqref{syst-2} if an only if 
\begin{equation}\label{Lax-Mil}
\forall v \in \mathcal{H}^1_{0} ( \Omega_1), \qquad Q(w, v ) :=  \int_{\Omega_1 } \hbar ^2  \nabla_x w \nabla_x \overline{v}+  iw \overline{v} (x) dx = T_F (v) := \int_{\Sigma} \hbar^2 F \overline{v} (x) d \sigma.
\end{equation}

From the trace theorem, the map 
$$ v \in \mathcal{H}^1_{0}( \Omega_1) \mapsto v \mid_{\Sigma} \in H^{\frac 1 2} ( \Sigma)$$
is continuous and hence for any $F\in H^{- \frac 1 2 } ( \Sigma)$, the map 
$$ v \mapsto T_F(v) \in \mathbb{C}$$ is a continuous antilinear form on $\mathcal{H}^1_{0}(\Omega_1).$

The existence of a unique solution to~\eqref{Lax-Mil} (and consequently the solution to~\eqref{syst-1}, ~\eqref{syst-2}) now follows from Lax-Milgram Theorem. 
Applying~\eqref{Lax-Mil} to $v = w$, we get 
$$ \| h \nabla_x w\|_{L^2( \Omega_1)} ^2 +\|  w\|_{L^2( \Omega_1)} ^2  \leq 2 |T_F(w)| \leq C \hbar^2 \| F\|_{H^{- \frac 1 2}( \Sigma)} \| w\|_{H^1(\Omega_1)},$$
which implies 
$$ \| w\| _{H^1( \Omega_1)} \leq C \| F\|_{H^{- \frac 1 2} ( \Sigma)}, $$ 
and using again~\eqref{Lax-Mil} with $v=w$, 
$$ \| w\| ^2_{L^2( \Omega_1)} \leq  \hbar^2 \| \nabla_x w\| _{L^2( \Omega_1)}  +\hbar^2 | T_F (w)| \leq C \hbar^2  \| F\|^2_{H^{- \frac 1 2} ( \Sigma)}.
$$ 
This proves the first part in Proposition~\ref{pb:mixed}. The proof of the second part is standard elliptic regularity results. Indeed,
we have 
$$ \Delta w = i \hbar^{-2} w, \qquad  \partial_\nu w \mid_\Sigma = F \in H^{\frac 1 2} ( \Sigma), \qquad  w \mid_{\partial \Omega_1 \setminus \Sigma} =0,$$
and we deduce by standard elliptic regularity results,
$$ \| w \|_{H^2( \Omega_1)} \leq C \Bigl(\hbar^{-2} \| w\|_{L^2( \Omega_1)} + \| F\|_{H^{\frac 1 2 } ( \Sigma)}\Bigr) \leq C \Bigl( \hbar^{-1} \| F\|_{H^{- \frac 1 2} ( \Sigma)} + \| F\|_{H^{\frac 1 2 } ( \Sigma)}\Bigr)
$$

This completes the proof of Proposition \ref{pb:mixed}.
\end{proof}

\subsection{Estimates for some operators}
\begin{lem}\label{technical1}
If $b(x,\xi)\in S^{-m}$ ($m\geq 0$) is compactly supported in $x\in\R^n$, then for any $s\in\R$, 
	$$\mathrm{Op}_h(b)=\mathcal{O}(h^{-\theta}): H^{s}(\R^n)\rightarrow H^{s+\theta}(\R^n),\quad \forall \theta\in[0,m].
	$$
\end{lem}
\begin{proof}
First we show that $\mathrm{Op}_h(b)$ is bounded from $H^s$ to $H^s$. It is equivalent to show that the operator $T_h:=\langle D_x\rangle^s\mathrm{Op}_h(b)\langle D_x\rangle^{-s}$ is bounded (independent of $h$) from $L^2$ to $L^2$. By definition, we have
\begin{align*}
\widehat{(T_hf)}(\xi)=\frac{1}{(2\pi)^d}\int_{\R^n}\langle\xi\rangle^s\widehat{b}(\xi-\eta,h\eta)\langle\eta\rangle^{-s}\widehat{f}(\eta)d\eta,
\end{align*}
where $\widehat{b}(\zeta,\eta)=(\mathcal{F}_{x\rightarrow\zeta}a)(\zeta,\eta)$ is a well-defined function. Thus $\widehat{T_hf}$ can be viewed as an operator acting on $\widehat{f}\in L^2(\R_{\xi}^d)$ with Schwartz kernel 
$$ K_h(\xi,\eta):=\frac{1}{(2\pi)^d}\langle\xi\rangle^s\langle\eta\rangle^{-s}\widehat{b}(\xi-\eta,h\eta).
$$ 
By Schur's test, to check the boundedness of this operator, it suffices to check that 
$$ \sup_{\xi,h}\int_{\R^n}|K_h(\xi,\eta)|d\eta<\infty,\quad \sup_{\eta,h}\int_{\R^n}|K_h(\xi,\eta)|d\xi<\infty.
$$ 
Since $K_h(\xi,\eta)$ is rapidly decaying in $\langle\xi-\eta\rangle$, these conditions can be simply verified by the elementary convolution inequalities:
\begin{align}\label{convolution1}
 \int_{\R^n} \frac{1}{\langle\eta\rangle^s\langle\xi-\eta\rangle^M}d\eta\leq C_M\langle\xi\rangle^{-s},\quad \forall M>d , s\geq 0,
\end{align}
and
\begin{align}\label{convolution2} \int_{\R^n}\frac{\langle\eta\rangle^{\sigma}}{\langle\xi-\eta\rangle^M}d\eta \leq C_{M,\sigma}\langle\xi\rangle^{\sigma},\quad \forall M>d+\sigma,\sigma\geq 0.
\end{align}
By interpolation, to finish the proof, it suffices to estimate the operator bound of $\mathrm{Op}_h(b)$ from  $H^s$ to $H^{s+m}$. Similarly, we need to check that the kernel
$$ G_h(\xi,\eta)=h^{m}\langle\xi\rangle^{s+m}\widehat{b}(\xi-\eta,h\eta)\langle\eta\rangle^{-s}
$$
satisfies the conditions for Schur's test.
First note that for any $\alpha\in\N^n$,
$$ (i(\xi-\eta))^{\alpha}\widehat{b}(\xi-\eta,\eta)=\frac{1}{(2\pi)^d}\int_{\R^n}(\partial_x^{\alpha}b)(x,\eta)e^{-ix\cdot(\xi-\eta)}dx,
$$
thus $\widehat{b}(\xi-\eta,h\eta)=O\big(\langle\xi-\eta\rangle^{-M}\langle h\eta\rangle^{-m}\big)$ for any $M\in\N$. Note that
$$\langle hm\rangle^{-m}\sim (1+h|\eta|)^{-m}\leq h^{-m}\langle\eta\rangle^{-m}.$$
This implies that
$$ |G_h(\xi,\eta)|\leq C_M \langle\xi\rangle^{s+m}\langle\eta\rangle^{-(s+m)}\langle\xi-\eta\rangle^{-M}.
$$
Now the boundeness of the  integration $\int G_h(\xi,\eta)d\eta$ or $\int G_h(\xi,\eta)d\xi$ follows from  the same convolution inequalities \eqref{convolution1} and \eqref{convolution2}. This completes the proof of Lemma \ref{technical1}.
\end{proof}

\begin{lem}\label{symbolic}
Let $a\in S^0(\R^{2n}), b\in S^0(\R^{2n})$ be two symbols with compact support in the $x$ variable. Then for any $N\in\N, N\geq 2n$, 
\begin{align*}
 &\Big\|\mathrm{Op}(a)\mathrm{Op}(b)-\sum_{|\alpha|\leq N}\frac{1}{i^{|\alpha|}\alpha !}\mathrm{Op}\big(\partial_{\xi}^{\alpha}a\partial_x^{\alpha}b \big) \Big\|_{\mathcal{L}(H^s\rightarrow H^s)}\\ \leq &C_{N}\sum_{|\beta|\leq K(n)}\sup_{|\alpha|=N+1}\sup_{(x,\xi)\in\R^{2n}}\iint_{\R^{2n}}\big|\partial_{x,\xi}^{\beta}\partial_z^{\alpha}\partial_{\zeta}^{\alpha}A(x,z,\xi,\zeta)\big|dz d\zeta, 
\end{align*} 
where
$$ A(x,x,\xi,\zeta)=a(x,\xi+\zeta)b(x+z,\xi).
$$
\end{lem}
\begin{proof}
The symbol of the operator $$\mathrm{Op}(a)\mathrm{Op}(b)-\sum_{|\alpha|\leq N}\frac{1}{i^{|\alpha|}\alpha !}\mathrm{Op}\big(\partial_{\xi}^{\alpha}a\partial_x^{\alpha}b \big)$$
 is given by
$$ r_N(x,\xi):=\frac{1}{N!}\iint _{\R^{2n}}\int_0^1(1-t)^N\sum_{|\alpha_1|+|\alpha_2|=N+1}(\partial_y^{\alpha_1}\partial_{\eta}^{\alpha_2}A)(x,tz,\xi,t\zeta)z^{\alpha_1}\zeta^{\alpha_2}\mathrm{e}^{-iz\cdot\zeta}dzd\zeta dt,
$$
with
$$ A(x,z,\xi,\zeta)=a(x,\xi+\zeta)b(x+z,\xi).
$$
Using the identity
$$ z^{\alpha_1}\zeta^{\alpha_2}\mathrm{e}^{-iz\cdot\zeta}=i^{N+1}\partial_{z}^{\alpha_2}\partial_{\zeta}^{\alpha_1}(\mathrm{e}^{-iz\cdot\zeta})
$$ 
and doing the integration by part, we have
\begin{align*} r_N(x,\xi)=&\sum_{|\alpha|=N+1}\frac{i^{N+1}}{N!}\int_0^1(1-t)^Nt^{N+1}dt\iint_{\R^{2n}}(\partial_z^{\alpha}\partial_{\zeta}^{\alpha}A)(x,tz,\xi,t\zeta)\mathrm{e}^{-iz\cdot\zeta}dzd\zeta\\
=&\sum_{|\alpha|=N+1}\frac{i^{N+1}}{N!}\int_0^1(1-t)^Nt^{N+1-2n}dt\iint_{\R^{2n}}(\partial_z^{\alpha}\partial_{\zeta}^{\alpha}A)(x,z,\xi,\zeta)\mathrm{e}^{-it^{-2}z\cdot\zeta}dzd\zeta
\end{align*}
Hence the integral converges absolutely.
Viewing $r_N(x,\xi)$ as a symbol of order $0$, we obtain the desired bound, thanks to the Caldr\'on-Vaillancourt theorem. 
\end{proof}

\subsection{Special sequence of eigenfunctions of a disc}
First we recall that
$$ J_m(z)=\Big(\frac{z}{2}\Big)^m\sum_{k=0}^{\infty}\frac{(-1)^k\big(\frac{z}{2}\big)^{2k}}{k!(m+k)!}
$$
are the Bessel functions satisfying the Bessel differential equation:
$$ z^2J_m''(z)+zJ_m'(z)+(z^2-m^2)J_m(z)=0.
$$
By definition, one has
\begin{equation}\label{recurrence}
J_{m+1}(z)+J_{m-1}(z)=\frac{2m}{z}J_m(z),\quad J_{m-1}(z)-J_{m+1}(z)=2J_m'(z).
\end{equation}

Denote by $\lambda_{m,n}$ the $n$-th zero of $J_m(z)$. It is well known that
$$ \lambda_{m,1}<\lambda_{m,2}<\cdots< \lambda_{m,n}<\cdots
$$
and the functions 
$$ \varphi_{m,n}(r,\theta)=J_m(\lambda_{m,n}r)\mathrm{e}^{im\theta}
$$
form an orthogonal sequence of eigenfunctions of $\Delta_{\mathbb{D}}$, associated with eigenvalues $\{\lambda_{m,n}^2:m\in\Z, n\in\N \}$. We will chose a special sequence
$$ J_{\alpha n}(\lambda_{\alpha n,n}r)\mathrm{e}^{i\alpha n\theta}
$$
for some $\alpha\in\N$, to be fixed later. Let us recall some facts about the zeros of Bessel functions:
\begin{prop}[\cite{E84}]\label{zeros}
There exists a continuous function $\iota:[-1,\infty)$, such that $$\lambda_{\alpha n,n}<n\iota(\alpha),\text{and } \lim_{n\rightarrow\infty}\frac{\lambda_{\alpha n,n}}{n}=\iota(\alpha).
$$
Moreover, there exists $0<\beta_1<\beta_2$, such that for all $\alpha\geq 1$,
$$ 1+\beta_1\alpha^{-\frac{2}{3}}<\frac{\iota(\alpha)}{\alpha}\leq 1+\beta_2\alpha^{-\frac{2}{3}}.
$$
\end{prop}


Thanks to this proposition, we have:
\begin{lem}\label{energynorm}
Fix $\alpha\in\N$, large enough and let
$$ w_n:=\frac{\varphi_{\alpha n,n}}{\lambda_{\alpha n,n}\|\varphi_{\alpha n,n}\|_{L^2(\mathbb{D})}}.
$$
Then we have
$$ \|(\partial_{\nu}w_n)|_{\partial\mathbb{D}}\|_{L^2(\partial\D)}=O(1),\quad  \mathrm{WF}_h(\partial_{\nu}w_h|_{\Sigma})\subset\mathcal{H}_{\delta}(\partial\mathbb{D}):=\{\delta<r_0<1-\delta \}
$$ 
where $h=(h_n)_{n\in\N}, h_n=\lambda_{\alpha n,n}^{-1}\sim (\iota(\alpha)n)^{-1}$ and the semiclassical wave-front set is taken for the sequence $(w_n)_{n\in\N}$, with a little abuse of the notation.
\end{lem}
\begin{proof}
To simplify the notation, we write $m=\alpha n$ and $\iota:=\iota(\alpha).$ From Proposition \ref{zeros}, we have
$$ 1+\beta_1\alpha^{-\frac{2}{3}}-o(1)<\frac{\iota}{\alpha}-o(1)=\frac{\lambda_{m,n}}{m}<\frac{\iota}{\alpha}\leq 1+\beta_2\alpha^{-\frac{2}{3}},\text{ as }n\rightarrow\infty.
$$
Note that at $r=1$, $\partial_{\nu}=\partial_r$ and $|\nabla w|^2=|\partial_r w|^2+\frac{1}{r^2}|\partial_{\theta}w|^2$. The hyperbolicity at the boundary is essentially due to the fact that
$$ \partial_{\theta}w_n=imw_n
$$
and 
\begin{equation}\label{i}
\frac{|m|}{\lambda_{m,n}}=\frac{\alpha}{\iota}+o(1)\leq 1-\delta(\alpha)
\end{equation} for $n\gg 1$. Let $0<\epsilon_0<\delta(\alpha)$, $\chi\in C^{\infty}(\R)$ such that $\chi(s)\equiv 0$ if $|s|>1-\epsilon_0$. From \eqref{i} we have $w_n=\chi(h_n\partial_{\theta})w_n.$ Since $\partial_{\theta}^2$ is just the Laplace operator on $L^2(\partial\mathbb{D})$, we have, near $\partial\mathbb{D}$, $\mathrm{WF}_h(w_n)$ is contained in $r>\epsilon_0>0$, thus $w_n$ is microlocalized near $\mathcal{H}(\Sigma)$. The estimate $\|\partial_{r}w_n|_{r=1}\|_{L^2(\partial\mathbb{D})}=O(1)$ then follows from the hyperbolicity and the fact that $\|\nabla w_n\|_{H^1(\mathbb{D})}=1$. This completes the proof of Lemma \ref{energynorm}.  
\end{proof}

\begin{center}
	\begin{tikzpicture}[scale=0.6]
\draw (0,0) circle [radius=2.828427];
\draw[blue] (2,2) rectangle (-2,-2); 
\draw[rotate around={2:(0,0)},blue] (2,2) rectangle (-2,-2); 
\draw[rotate around={4:(0,0)},blue] (2,2) rectangle (-2,-2);
\draw[rotate around={6:(0,0)},blue] (2,2) rectangle (-2,-2); 
\draw[rotate around={8:(0,0)},blue] (2,2) rectangle (-2,-2); 
\draw[rotate around={10:(0,0)},blue] (2,2) rectangle (-2,-2); 
\draw[rotate around={12:(0,0)},blue] (2,2) rectangle (-2,-2); 
\draw[rotate around={14:(0,0)},blue] (2,2) rectangle (-2,-2); 
\draw[rotate around={16:(0,0)},blue] (2,2) rectangle (-2,-2); 
\draw[rotate around={18:(0,0)},blue] (2,2) rectangle (-2,-2); 
\draw[rotate around={20:(0,0)},blue] (2,2) rectangle (-2,-2); 
\draw[rotate around={22:(0,0)},blue] (2,2) rectangle (-2,-2);  
\draw[rotate around={24:(0,0)},blue] (2,2) rectangle (-2,-2); 
\draw[rotate around={26:(0,0)},blue] (2,2) rectangle (-2,-2); \draw[rotate around={28:(0,0)},blue] (2,2) rectangle (-2,-2); \draw[rotate around={30:(0,0)},blue] (2,2) rectangle (-2,-2); \draw[rotate around={32:(0,0)},blue] (2,2) rectangle (-2,-2); \draw[rotate around={34:(0,0)},blue] (2,2) rectangle (-2,-2); \draw[rotate around={36:(0,0)},blue] (2,2) rectangle (-2,-2); \draw[rotate around={38:(0,0)},blue] (2,2) rectangle (-2,-2); \draw[rotate around={40:(0,0)},blue] (2,2) rectangle (-2,-2); \draw[rotate around={42:(0,0)},blue] (2,2) rectangle (-2,-2); 
\draw[rotate around={44:(0,0)},blue] (2,2) rectangle (-2,-2); \draw[rotate around={46:(0,0)},blue] (2,2) rectangle (-2,-2); \draw[rotate around={48:(0,0)},blue] (2,2) rectangle (-2,-2); \draw[rotate around={50:(0,0)},blue] (2,2) rectangle (-2,-2); \draw[rotate around={52:(0,0)},blue] (2,2) rectangle (-2,-2); \draw[rotate around={54:(0,0)},blue] (2,2) rectangle (-2,-2); \draw[rotate around={56:(0,0)},blue] (2,2) rectangle (-2,-2); \draw[rotate around={58:(0,0)},blue] (2,2) rectangle (-2,-2); \draw[rotate around={60:(0,0)},blue] (2,2) rectangle (-2,-2); \draw[rotate around={62:(0,0)},blue] (2,2) rectangle (-2,-2); \draw[rotate around={64:(0,0)},blue] (2,2) rectangle (-2,-2); \draw[rotate around={68:(0,0)},blue] (2,2) rectangle (-2,-2);
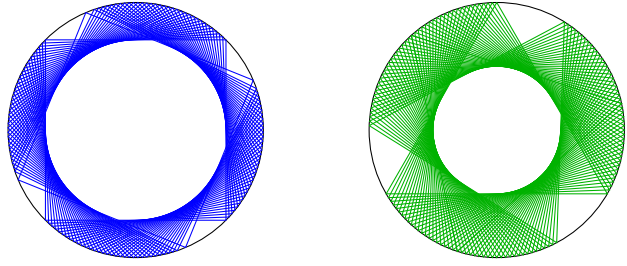
\draw[black] (5,-4) node{Concentration of the eigenfunctions $\varphi_{\alpha n,n}$ as $n\rightarrow\infty$ };
\draw (8,0) circle [radius=2.828427];
\draw[green!70!black] (8,2.828427)--(5.55051, -1.41421)--
(10.4494897,-1.41421)--(8,2.828427);
\draw[rotate around={2:(8,0)} ,green!70!black] (8,2.828427)--(5.55051, -1.41421)--
(10.4494897,-1.41421)--(8,2.828427);
\draw[rotate around={4:(8,0)} ,green!70!black] (8,2.828427)--(5.55051, -1.41421)--
(10.4494897,-1.41421)--(8,2.828427);\draw[rotate around={6:(8,0)} ,green!70!black] (8,2.828427)--(5.55051, -1.41421)--
(10.4494897,-1.41421)--(8,2.828427);\draw[rotate around={8:(8,0)} ,green!70!black] (8,2.828427)--(5.55051, -1.41421)--
(10.4494897,-1.41421)--(8,2.828427);\draw[rotate around={10:(8,0)} ,green!70!black] (8,2.828427)--(5.55051, -1.41421)--
(10.4494897,-1.41421)--(8,2.828427);\draw[rotate around={12:(8,0)} ,green!70!black] (8,2.828427)--(5.55051, -1.41421)--
(10.4494897,-1.41421)--(8,2.828427);
\draw[rotate around={14:(8,0)} ,green!70!black] (8,2.828427)--(5.55051, -1.41421)--
(10.4494897,-1.41421)--(8,2.828427);\draw[rotate around={16:(8,0)} ,green!70!black] (8,2.828427)--(5.55051, -1.41421)--
(10.4494897,-1.41421)--(8,2.828427);\draw[rotate around={18:(8,0)} ,green!70!black] (8,2.828427)--(5.55051, -1.41421)--
(10.4494897,-1.41421)--(8,2.828427);\draw[rotate around={20:(8,0)} ,green!70!black] (8,2.828427)--(5.55051, -1.41421)--
(10.4494897,-1.41421)--(8,2.828427);\draw[rotate around={22:(8,0)} ,green!70!black] (8,2.828427)--(5.55051, -1.41421)--
(10.4494897,-1.41421)--(8,2.828427);\draw[rotate around={24:(8,0)} ,green!70!black] (8,2.828427)--(5.55051, -1.41421)--
(10.4494897,-1.41421)--(8,2.828427);\draw[rotate around={26:(8,0)} ,green!70!black] (8,2.828427)--(5.55051, -1.41421)--
(10.4494897,-1.41421)--(8,2.828427);\draw[rotate around={28:(8,0)} ,green!70!black] (8,2.828427)--(5.55051, -1.41421)--
(10.4494897,-1.41421)--(8,2.828427);\draw[rotate around={30:(8,0)} ,green!70!black] (8,2.828427)--(5.55051, -1.41421)--
(10.4494897,-1.41421)--(8,2.828427);\draw[rotate around={32:(8,0)} ,green!70!black] (8,2.828427)--(5.55051, -1.41421)--
(10.4494897,-1.41421)--(8,2.828427);\draw[rotate around={34:(8,0)} ,green!70!black] (8,2.828427)--(5.55051, -1.41421)--
(10.4494897,-1.41421)--(8,2.828427);\draw[rotate around={36:(8,0)} ,green!70!black] (8,2.828427)--(5.55051, -1.41421)--
(10.4494897,-1.41421)--(8,2.828427);\draw[rotate around={38:(8,0)} ,green!70!black] (8,2.828427)--(5.55051, -1.41421)--
(10.4494897,-1.41421)--(8,2.828427);\draw[rotate around={40:(8,0)} ,green!70!black] (8,2.828427)--(5.55051, -1.41421)--
(10.4494897,-1.41421)--(8,2.828427);\draw[rotate around={42:(8,0)} ,green!70!black] (8,2.828427)--(5.55051, -1.41421)--
(10.4494897,-1.41421)--(8,2.828427);\draw[rotate around={42:(8,0)} ,green!70!black] (8,2.828427)--(5.55051, -1.41421)--
(10.4494897,-1.41421)--(8,2.828427);\draw[rotate around={44:(8,0)} ,green!70!black] (8,2.828427)--(5.55051, -1.41421)--
(10.4494897,-1.41421)--(8,2.828427);\draw[rotate around={46:(8,0)} ,green!70!black] (8,2.828427)--(5.55051, -1.41421)--
(10.4494897,-1.41421)--(8,2.828427);\draw[rotate around={48:(8,0)} ,green!70!black] (8,2.828427)--(5.55051, -1.41421)--
(10.4494897,-1.41421)--(8,2.828427);\draw[rotate around={50:(8,0)} ,green!70!black] (8,2.828427)--(5.55051, -1.41421)--
(10.4494897,-1.41421)--(8,2.828427);
\draw[rotate around={52:(8,0)} ,green!70!black] (8,2.828427)--(5.55051, -1.41421)--
(10.4494897,-1.41421)--(8,2.828427);\draw[rotate around={54:(8,0)} ,green!70!black] (8,2.828427)--(5.55051, -1.41421)--
(10.4494897,-1.41421)--(8,2.828427);\draw[rotate around={56:(8,0)} ,green!70!black] (8,2.828427)--(5.55051, -1.41421)--
(10.4494897,-1.41421)--(8,2.828427);\draw[rotate around={58:(8,0)} ,green!70!black] (8,2.828427)--(5.55051, -1.41421)--
(10.4494897,-1.41421)--(8,2.828427);\draw[rotate around={60:(8,0)} ,green!70!black] (8,2.828427)--(5.55051, -1.41421)--
(10.4494897,-1.41421)--(8,2.828427);
\draw[rotate around={62:(8,0)} ,green!70!black] (8,2.828427)--(5.55051, -1.41421)--
(10.4494897,-1.41421)--(8,2.828427);\draw[rotate around={64:(8,0)} ,green!70!black] (8,2.828427)--(5.55051, -1.41421)--
(10.4494897,-1.41421)--(8,2.828427);\draw[rotate around={66:(8,0)} ,green!70!black] (8,2.828427)--(5.55051, -1.41421)--
(10.4494897,-1.41421)--(8,2.828427);\draw[rotate around={68:(8,0)} ,green!70!black] (8,2.828427)--(5.55051, -1.41421)--
(10.4494897,-1.41421)--(8,2.828427);\draw[rotate around={70:(8,0)} ,green!70!black] (8,2.828427)--(5.55051, -1.41421)--
(10.4494897,-1.41421)--(8,2.828427);
\draw[rotate around={72:(8,0)} ,green!70!black] (8,2.828427)--(5.55051, -1.41421)--
(10.4494897,-1.41421)--(8,2.828427);\draw[rotate around={74:(8,0)} ,green!70!black] (8,2.828427)--(5.55051, -1.41421)--
(10.4494897,-1.41421)--(8,2.828427);\draw[rotate around={76:(8,0)} ,green!70!black] (8,2.828427)--(5.55051, -1.41421)--
(10.4494897,-1.41421)--(8,2.828427);\draw[rotate around={78:(8,0)} ,green!70!black] (8,2.828427)--(5.55051, -1.41421)--
(10.4494897,-1.41421)--(8,2.828427);\draw[rotate around={80:(8,0)} ,green!70!black] (8,2.828427)--(5.55051, -1.41421)--
(10.4494897,-1.41421)--(8,2.828427);
\draw[rotate around={82:(8,0)} ,green!70!black] (8,2.828427)--(5.55051, -1.41421)--
(10.4494897,-1.41421)--(8,2.828427);\draw[rotate around={84:(8,0)} ,green!70!black] (8,2.828427)--(5.55051, -1.41421)--
(10.4494897,-1.41421)--(8,2.828427);\draw[rotate around={86:(8,0)} ,green!70!black] (8,2.828427)--(5.55051, -1.41421)--
(10.4494897,-1.41421)--(8,2.828427);\draw[rotate around={88:(8,0)} ,green!70!black] (8,2.828427)--(5.55051, -1.41421)--
(10.4494897,-1.41421)--(8,2.828427);

	\end{tikzpicture}
\end{center}

\subsection{Geometric optics construction}
In this part we recall the geometric optics construction for the hyperbolic boundary value problem. In the tubular neighborhood of the interface $\Sigma$, we use the geodesic normal coordinate $x=(y,x')$, such
$$ \Delta=\frac{1}{\kappa}\partial_y(\kappa\partial_y)+\frac{1}{\kappa}\partial_i(g_0^{ij}\kappa\partial_j),
$$ 
where $\kappa=\sqrt{\det(g_0)}$ and $\partial_j=\partial_{x_j'}$. The semiclassical operator
$$ P_{h}=h^2\Delta_{g_0}+1=h^2\partial_y^2+h^2g_0^{ij}\partial_i\partial_j+1+\frac{h}{\kappa}(\partial_y\kappa) h\partial_y+\frac{h}{\kappa}\partial_i(g_0^{ij}\kappa)h\partial_j.
$$ 
Let $f_0^{\pm}\in L^2(\R_{x'}^{d-1})$ such that $\mathrm{WF}_h(f_0^{\pm})$ lies in a neighborhood of $(y=0,x_0';\eta=0,\xi_0')$, such that
$$ r_0(0,x_0',\xi_0')\geq c_0>0.
$$ Denote by $\theta^{\pm}(\xi)=\mathcal{F}_h(\chi f_0^{\pm})(\xi)$, where $\chi\in C_c^{\infty}(\R_{x'}^{d-1})$, supported near $x_0'$.

Consider the semi-classical Fourier integral operators $U^{\pm}$, represented by
$$ U^{\pm}(\chi f_0^{\pm})(y,x')=\frac{1}{(2\pi h)^{d-1}}\int_{\R^{d-1}} \mathrm{e}^{\frac{i}{h}\varphi^{\pm}(y,x',\xi')}b^{\pm}(y,x',\xi')\theta^{\pm}(\xi')d\xi'.
$$
We have
\begin{align*}
P_h(U^{\pm}(\chi f_0^{\pm}))=\frac{1}{(2\pi h)^{d-1}}\int_{\R^{d-1}} (h^2\Delta_g+1)(\mathrm{e}^{\frac{i\varphi^{\pm}}{h}} b^{\pm})\theta^{\pm}(\xi')d\xi'.
\end{align*}
Observing that
\begin{align*}
(h^2\Delta_{g_0}+1)(\mathrm{e}^{\frac{i\varphi^{\pm}}{h}} b^{\pm})=&(1-|\nabla_{g_0}\varphi^{\pm}|^2)b^{\pm}\mathrm{e}^{\frac{i\varphi^{\pm}}{h}}
+ih(2\nabla_{g_0}\varphi^{\pm}\cdot\nabla_{g_0}b^{\pm}+\Delta_{g_0}\varphi^{\pm}b^{\pm}  )\mathrm{e}^{\frac{i\varphi^{\pm}}{h}}\\
+& h^2(\Delta_{g_0}b^{\pm})\mathrm{e}^{\frac{i\varphi^{\pm}}{h}}.
\end{align*}
Near $\mathrm{WF}_h(f_0)$ and for small $y$, we can solve the eikonal equation
\begin{equation}\label{eikonal}  1-|\nabla_{g_0}\varphi^{\pm}|^2=0,\quad \varphi^{\pm}|_{y=0}(x',\xi')=x'\cdot\xi'.
\end{equation}
Note that
$ |\nabla_{g_0}\varphi^{\pm}|^2=|\partial_y\varphi^{\pm}|^2+g_0^{jk}\partial_j\varphi^{\pm}\partial_k\varphi^{\pm}.
$ Near $(x_0',\xi_0')\in\mathcal{H}(\Sigma)$, for each fixed $\xi'$, we find a Lagrangian submanifold of $T^*\Sigma$, locally of the form
$$ \mathcal{L}_{0,\xi'}:=\{(x',\xi=\partial_{x'}\varphi_0(x,\xi')): \varphi_0(x',\xi')=x'\cdot\xi' \}.
$$
At each point $(x',\xi=\xi')\in\mathcal{L}_{0,\xi'}$, there are two distinct roots $\eta^{\pm}$ of the equation
$$ \eta^2+g_0^{jk}\xi_j'\xi_k'=1,
$$
and each root determines a flow $\Phi_y^{\pm}$ of the bicharacteristics $p=\eta^2-r(y,x',\xi')$ on $\{p=0\}$. Then we can define the Lagrangian
$ \mathcal{L}_{y,\xi'}^{\pm}:=\exp(\Phi_y^{\pm})(\mathcal{L}_{0,\xi'}) 
$ locally, which is again a Lagrangian of $T^*\Sigma$ (viewing $y$ as a parameter) and can be written locally as $\mathcal{L}_{y,\xi'}^{\pm}=\{(x',\partial_{x'}\varphi^{\pm})\}$. Then $\varphi^{\pm}$ is the desired solutions of \eqref{eikonal} with the property
$$ \partial_y\varphi^{+}+\partial_y\varphi^{-}=0,\text{ at }y=0.
$$
Next we set
$$ b^{\pm}(y,x',\xi')=\sum_{j=0}^Nh^jb_j^{\pm}(y,x',\xi'),
$$
with coefficients $b_j$ 
solving the transport equations
\begin{equation}\label{transport}
\begin{split} 
 &2\partial_y\varphi^{\pm}\partial_y b_0^{\pm}+g_0^{jk}\partial_j\varphi^{\pm}\partial_kb_0^{\pm}+(\Delta_{g_0}\varphi^{\pm})b_0^{\pm}=0,\\
 &2i\partial_y\varphi^{\pm}\partial_y b_j^{\pm}+ig_0^{jk}\partial_j\varphi^{\pm}\partial_kb_j^{\pm}+i(\Delta_{g_0}\varphi^{\pm})b_{j}^{\pm}+\Delta_{g_0}b_{j-1}^{\pm}=0,\; 1\leq j\leq N.
 \end{split}
\end{equation}
Then 
$$ P_h(\chi f_0^{\pm})=\frac{h^{N+2}}{(2\pi h)^{d-1}}\int_{\R^{d-1}}\mathrm{e}^{\frac{i\varphi^{\pm}}{h}}\Delta_{g_0}b_N^{\pm}(y,x',\xi')\theta^{\pm}(\xi')d\xi'=O_{L^2}(h^{N+2}).
$$
To determine the datum $b_j^{\pm}|_{y=0}$, we need the boundary conditions. Note that 
 the approximate quasi-mode is given by
$$ u_h=\sum_{\pm} U^{\pm}(\chi f_0^{\pm})
$$ 
and we want to determine $f_0^{\pm}$.

Denote by  $B_h^{\pm}=\mathrm{Op}_h(b^{\pm})$ and $B_h^{0,\pm}=\mathrm{Op}_h(b^{\pm}|_{y=0})$,
then the Dirichlet trace is given by
$$ B_h^{0,+}(\chi f_0^{+})+B_h^{0,-}(\chi f_0^{-}),
$$
and the Neumann trace is given by 
$$ \sum_{\pm}\pm\mathrm{Op}_h(\sqrt{r_0}b^{\pm}|_{y=0} )(\chi f_0^{\pm})+h\sum_{\pm}\mathrm{Op}_h(\partial_yb^{\pm}|_{y=0} )|_{y=0}(\chi f_0^{\pm}).
$$
Now we choose $b_0^{+}|_{y=0}=b_0^-|_{y=0}=\chi(x')\psi(\xi')$ localized near $(x_0',\xi_0')$ and $b_j^{\pm}|_{y=0}=0$ for all $1\leq j\leq N$. Then the symbol (matrix-valued)
\begin{align*}
\Theta=\Theta_0+h
\left(\begin{matrix}
0  &0\\
\partial_yb_0^+ &\partial_yb_0^-
\end{matrix}
\right)|_{y=0}
\end{align*}
with
$$\Theta_0:=\left( 
\begin{matrix}
b_0^+  &b_0^-\\
\sqrt{r_0}b_0^+ &-\sqrt{r_0}b_0^-
\end{matrix}\right)|_{y=0}
$$
is invertible. For such an elliptic system, we can construct a symbol (matrix-valued) $\Upsilon$, such that
$$ \mathrm{Op}_h(\Theta)\mathrm{Op}_h(\Upsilon)=\mathrm{Id}+\mathcal{O}_{H^s\rightarrow H^{s+m}}(h^{N+1-m}).
$$
In particular, for a given Dirichlet trace $\sigma_{Dir}$ and Neumann trace $\sigma_{Neu}$ with wave front sets located near $(x_0',\xi_0')$, we find
$$\binom{\chi f_0^+}{\chi f_0^-}=\chi\mathrm{Op}_h(\Upsilon)\chi\binom{\sigma_{Dir}}{\sigma_{Neu}}.
$$
Then microlocally near $(x_0',\xi_0')\in\mathcal{H}(\Sigma)$, $u_h$ satisfies
$$ (h^2\Delta_{g_0}+1)u_h=O_{L^2}(h^N),\quad u|_{h=0}=\sigma_{Dir}+O_{H^{\frac{1}{2}}}(h^{N}),\; h\partial_yu_h|_{y=0}=\sigma_{Neu}+O_{H^{-\frac{1}{2}}}(h^N),
$$
and microlocally near $(x_0',\xi_0')$, $u_h=O_{L^2}(1)$,$\mathrm{WF}_h(u_h)$ lies in a small neighborhood of $(x_0',\xi_0')$. 
Finally, due to the microlocalisation in the hyperbolic region, we can exchange in the error terms powers of $h$ against derivatives, leading to 
$$ (h^2\Delta_{g_0}+1)u_h=O_{H^{k}}(h^{N-k}),\; u|_{h=0}=\sigma_{Dir}+O_{H^{\frac{1}{2} +k }}(h^{N-k }),\; h\partial_yu_h|_{y=0}=\sigma_{Neu}+O_{H^{-\frac{1}{2} + k }}(h^{N-k}).
$$




\end{document}